\documentclass[11pt]{article}

\def\UseSection{%%
        \numberwithin{equation}{section}
        \newtheorem{theorem}    {Theorem}[section]
        \DefineTheorems % Use this to define other environments to be
}
\usepackage[textwidth=500pt,textheight=650pt,centering]{geometry} % 11pt

\usepackage{amsfonts}
\usepackage{amsmath,amssymb,amsthm}
\usepackage{appendix}
\usepackage{bbm} % used in \newcommand{\1}{\mathbbm{1}}
\usepackage{amsbsy}
\usepackage{enumerate}
\usepackage{cite}

\usepackage[bookmarks, colorlinks, breaklinks,
pdftitle={},pdfauthor={}]{hyperref}

\hypersetup{
  linkcolor=black,
  citecolor=black,
  filecolor=black,
  urlcolor=black}

\usepackage{verbatim}
\usepackage{tikz}
\usepackage{color}

\newcommand{\black}{\black}

\usepackage[font=small,labelfont=bf]{caption}

\numberwithin{equation}{section}

\usetikzlibrary{calc,decorations.markings}
\usetikzlibrary{decorations.pathmorphing}
\usetikzlibrary{decorations.pathreplacing}
\usetikzlibrary{patterns}
\usetikzlibrary{arrows}

\newcommand{\bb}[1]{\mathbb{#1}}

\newcommand{\1}{\mathbbm{1}}

\newcommand{\floor}[1]{\lfloor #1 \rfloor}

\newcommand{\D}{\mathrm{d}}

\newcommand{\blank}[1]{}

\newcommand{\E}{\bb E}
\newcommand{\R}{\bb R}
\newcommand{\Z}{\bb Z}

\newcommand{\N}{\bb N}

\renewcommand{\P}{\bb P}

\DeclareMathOperator{\arccosh}{arccosh}
\DeclareMathOperator{\arcsinh}{arcsinh}

\newcommand{\nnb}	{\nonumber \\}

\def\DefineTheorems{%%
	\newtheorem{lemma}      [theorem] {Lemma}

	\newtheorem{prop}        [theorem] {Proposition}
	\newtheorem{defn}       [theorem] {Definition}
	\theoremstyle{definition}% ``defn'' theorem style
	
}

\UseSection
\newcommand{\q}{q}

\newcommand{\scale}{s}
\newcommand{\xtop}{r}

 \title {
   Asymptotic behaviour of the lattice Green function
 }

 \author{
    Emmanuel Michta\thanks{Department of Mathematics,
     University of British Columbia,
     Vancouver, BC, Canada V6T 1Z2.
     Michta: \url{https://orcid.org/0000-0001-7222-0422}, {\tt michta@math.ubc.ca}.
     Slade:  \url{https://orcid.org/0000-0001-9389-9497}, {\tt slade@math.ubc.ca}.}
    \and
   Gordon Slade$^*$}

 \date{\vspace{-5ex}} %for arXiv

\begin{document}
\maketitle

\begin{abstract}
The lattice Green function, i.e., the resolvent of the discrete Laplace
operator, is fundamental in probability theory and
mathematical physics.  We derive its long-distance behaviour
via a detailed analysis of 
an integral representation involving modified Bessel functions.
Our emphasis is on the
decay of the massive lattice Green function in the vicinity of the massless (critical)
case, and the recovery of Euclidean isotropy in the massless limit.
This provides a prototype for the expected but unproven long-distance behaviour of
near-critical two-point functions in statistical mechanical models such as
percolation, the Ising model, and the self-avoiding walk above their upper critical dimensions.
\end{abstract}

%\noindent
%Keywords:  random walk,
%lattice Green function, resolvent of discrete Laplacian, massive lattice propagator,
%Orstein--Zernike decay, modified Bessel function, Laplace method.
%
%\medskip \noindent
%MSC2010 Classifications: 33C10, 41A60, 60K35, 82B27.

\begin{center}
    {\it Dedicated to the memory of Dmitry Ioffe, 1963--2020.}
\end{center}

\section{The lattice Green function and its decay}

\subsection{Introduction}
\label{sec:intro}

The \emph{lattice Green function}
is defined to be the Fourier integral
\begin{equation}
\label{e:Cax}
     C_{a}(x) =
     \int_{[-\pi,\pi]^d} \frac{e^{ik\cdot x}}{a^2 +1 - \hat D(k)} \frac{\D k}{(2\pi)^d}
\end{equation}
with $x \in \Z^d$, $a \ge 0$, and
\begin{equation}
    \hat D(k) = \frac 1d \sum_{j=1}^d \cos k_j  \qquad  (k = (k_1,\ldots,k_d) ).
\end{equation}
The integral \eqref{e:Cax}
converges for all $a >0$ in all dimensions $d \ge 1$, but converges for
$a=0$ only for the transient case
$d>2$ since the denominator of the integrand is quadratic in small $k$ when $a=0$.
The integral $C_0(0)$ in dimension $d=3$ is a Watson integral that has been evaluated
explicitly
in \cite{Wats39}; see \cite{Zuck11,Gutt10} for the interesting history and further
developments.

The lattice Green function derives its name from the fact that it is equal to the $0,x$ matrix element
of $(-\Delta+a^2)^{-1}$
(the inverse is as an operator on $\ell_2(\Z^d)$).
Here $\Delta$ is the Laplacian on $\Z^d$, i.e.,
\begin{equation}
\label{e:DeltaJ}
    \Delta = D- I
\end{equation}
where $D$ is the $\Z^d \times \Z^d$ matrix with $D_{xy}=\frac{1}{2d}$ if $\|x-y\|_1=1$,
and otherwise $D_{xy}=0$, and $I$ denotes the identity matrix.
Thus $C_a$ is the resolvent of the lattice Laplacian.
In the physics literature,
$C_a$ is often called the \emph{Euclidean lattice scalar propagator}.

Our purpose is to study the precise asymptotic behaviour of $C_a(x)$ as $x$ goes
to infinity, with emphasis on how this behaviour depends on values of $a$
close to the critical value $a =0$, and on how Euclidean invariance is restored in the
small $a$ limit.
Let $m_a$ denote the unique nonnegative solution to $\cosh m_a = 1+da^2$.
For $a>0$, an elementary proof that $m_a$ is the
exponential rate of decay of $C_{a}(x)$ when $x\to\infty$ along a coordinate axis,
and that
$C_{a}(x) \le C_{a}(0)e^{-m_a\|x\|_\infty}$ for all $x \in \Z^d$, is given in
\cite[Theorem~A.2]{MS93} (with a change of notation $2dz=\frac{1}{1+a^2}$).
On the other hand, for the critical case $a=0$ it is well-known
that there is instead polynomial
decay $\|x\|_2^{-(d-2)}$ \cite{LL10}.
We will prove, in a unified way and with precise constants for
the amplitudes in the asymptotic formulas, that there are the following four regimes of decay
for dimensions $d \ge 1$ (with the restriction $d>2$ for regime (IV)):
\begin{alignat*}{4}
    & {\rm (I)}  &\qquad \text{fixed }a >0
    & \qquad\text{anisotropic OZ}
    & \qquad  m_a^{(d-3)/2} |x|_a^{-(d-1)/2}e^{-m_a|x|_a}
    \\
    &{\rm (II)}
    &\qquad a\|x\|_2\to \infty, \; a^3\|x\|_2\to 0
    &\qquad  \text{isotropic OZ}
    &\qquad a^{(d-3)/2} \|x\|_2^{-(d-1)/2}e^{-\sqrt{2d} a\|x\|_2}
    \\
    &{\rm (III)}
    &\qquad \text{fixed }a\|x\|_2 >0
    &\qquad  \text{massive continuum}
    &\qquad \|x\|_2^{-(d-2)}
    \\
    &{\rm (IV)}
    &\qquad a = 0 \;\;(d>2)
    &\qquad  \text{massless continuum}
    &
    \qquad \|x\|_2^{-(d-2)}
    .
\end{alignat*}

The decay in regime (I) is called \emph{Ornstein--Zernike decay}.
The norm $|\cdot|_a$ is an explicitly defined
$a$-dependent anisotropic norm on $\R^d$ (not an $\ell^p$ norm).
Ornstein--Zernike decay is widely studied and has been proved for a variety
of subcritical statistical mechanical models, e.g., \cite{CIV08,CC86b,CCC91}.
These proofs for much
more difficult models than the lattice Green function
show decay of the form $|x|_a^{-(d-1)/2}e^{-m_a|x|_a}$,
but do not however reveal
the factor $m_a^{(d-3)/2}$ for small $a>0$.  Indeed, a solution to the latter
problem would be tantamount to a control of the critical behaviour of those models, a topic with
difficult unsolved problems of great current interest.

The decay in regime (II) is also Ornstein--Zernike decay, but the mass in regime (I)
is now replaced by its asymptotic form $\sqrt{2d} a$ as $a \to 0$, and Euclidean
invariance is restored since the norm
$|x|_a$ from regime (I) is replaced by the Euclidean norm.
This is natural: we will prove that $\lim_{a \to 0}|x|_a
=\|x\|_2$.

The decay in regimes (III) and (IV) is in fact expressed in terms of the
continuum Green function for the Laplacian on $\R^d$: the massive Green function
in regime (III) and massless Green function in regime (IV).
These continuum Green functions appear explicitly in the full asymptotic formulas
in these regimes.  In both cases,
the decay is Euclidean invariant and is expressed in terms of the $\ell_2$ norm.

The transition from regime (II) to (III) can be
anticipated
by replacing $a$ in (II)
by $\|x\|_2^{-1}$, which corresponds to $x$ on the order of the
\emph{correlation length} $m_a^{-1}$.
This replacement causes the asymptotic formula in (II) to transform
into the formula in (III).

More generally, for real numbers $\q\in (0,\infty)$ we consider the decay of
\begin{equation}
\label{e:Cxq}
    C^{(\q)}_{a}(x)=
    \int_{[-\pi,\pi]^d}\frac{e^{ik\cdot x}}{(a^2+1-\hat{D}(k))^\q}\frac{\D k}{(2\pi)^d}
\end{equation}
in the above four regimes.
When $\q \ge 2$ is a positive integer, $C^{(\q)}_{a}$ is the $\q$-fold convolution of
$C_a$ with itself.
For $\q=2,3,4$, $C_a^{(\q)}(0)$ is known respectively as the bubble, triangle and square
diagram.  These diagrams play an important role in the study of various statistical mechanical
models above their upper critical dimensions, especially when $a=0$; see, e.g., \cite{Slad06}.
For  integers $\q \ge 2$, $C^{(\q)}_{0}(x)$  is the critical lattice polyharmonic
Green function.  Polyharmonic functions have been widely studied, especially on $\R^d$
rather than on the lattice $\Z^d$ (e.g., \cite{ACL83}).

We note in passing
that the lattice Green function has the following probabilistic interpretation.
Let $X_1,X_2,\ldots$ be independent and identically distributed
random variables with each $X_i$
equally likely to be any one of the $2d$ unit vectors (positive or negative) in $\Z^d$,
for any fixed integer $d \ge 1$.
For $a \in [0,\infty)$,
let $N$ be a geometric random variable with
\begin{equation}
    \P(N=n)=
    \Big(\frac{1}{1+a^2}\Big)^{n}
    \frac{a^2}{1+a^2}  \qquad (n\ge 0),
\end{equation}
with
$N$ independent of the $X_i$.  Then $\P(N \ge n)=(\frac{1}{1+a^2})^n$.
 Let $S_0=0$, and
consider the nearest-neighbour
random walk $S_n=X_1+\cdots +X_n$ on $\Z^d$ subjected to $a$-dependent killing, i.e.,
the walk takes $N$ steps and then dies.
Let $p_n(x)$ denote the probability that the random walk without killing makes a transition
from $0$ to $x$ in $n$ steps.
The expected number
of visits of the random walk to a point $x\in \Z^d$ is
\begin{align}
\label{e:Gkappadef}
     \E \Bigg(\sum_{n=0}^N \1_{S_n  =x} \Bigg)
     =
     \sum_{n=0}^\infty \Big(\frac{1}{1+a^2}\Big)^{n} p_n(x)
    \qquad (x\in \Z^d).
\end{align}
The expectation in \eqref{e:Gkappadef}
is equal to $(1+a^2) C_a^{(1)}(x)$, as
a consequence of the fact that the Fourier transform
of $p_n(x)$ is simply
\begin{equation}
    \sum_{x\in \Z^d}p_n(x) e^{ik\cdot x} = \hat D(k)^n
    .
\end{equation}
We do not consider more
general random walks, which would correspond to operators
other than the Laplacian.
We expect that our results should extend to
the Green function
for random walks taking finite-range symmetric steps, but as can be seen in \cite{AIOV21}
the nature of the decay can change when arbitrarily long steps are permitted.

Our motivation to study the decay of the lattice
Green function originates from statistical mechanics.
The long-distance asymptotic behaviour of the two-point function is
an essential feature in the analysis of
critical phenomena in lattice statistical mechanical models such as percolation, the Ising
model, or the self-avoiding walk.
In high dimensions, $\|x\|_2^{-(d-2)}$ decay of the critical two-point function
has been proved in several cases, including \cite{HHS03,Hara08,BHH21,Saka07}.
However, the near-critical behaviour, which merges the subcritical
exponential decay and the power-law critical decay, has received scant
attention despite the fact that it has the potential
to reveal important and hitherto unstudied aspects of the critical behaviour,
particularly for models defined on a torus.  Recently progress has been made
in this direction for weakly self-avoiding walk for dimensions $d>4$
\cite{MS22,Slad20_wsaw} and percolation for $d>6$
\cite{HMS22}.
In high dimensions, where mean-field behaviour is known to occur,
the near-critical two-point function is conjectured to have similar decay to that of
the lattice Green function.
It is therefore important to have a detailed understanding of the
long-distance behaviour of the lattice Green function
as a prototype.  In this paper, we provide a comprehensive account of the decay
of the lattice Green function.

\subsection{The anisotropic norm}

Lattice effects play a significant role
in the asymptotic behaviour of $C_a^{(\q)}(x)$ when
$a>0$ is fixed, and lead to anisotropy in the decay.
The
following definition enters into the description of the anisotropy.

\begin{defn}
\label{def:norm}
Let $d \ge 1$ and $a \ge 0$.
We define the \emph{mass}, or \emph{inverse correlation length}, to be the unique solution
$m_a\ge 0$ of
\begin{equation}
\label{e:m0def}
    \cosh m_a = 1+ da^2 .
\end{equation}
For nonzero $x \in \R^d$, let $u=u_a(x)\ge 0$ be the unique solution of
	\begin{equation}
\label{e:udef1}
		\frac{1}{d}\sum_{i=1}^d\sqrt{1+x_i^2u^2} = 1+a^2,
	\end{equation}
which exists since the left-hand side of \eqref{e:udef1} is a strictly
increasing function of $u \in [0,\infty)$ onto $[1,\infty)$.
Finally, with the restriction now that
$a>0$, we define
$|0|_a=0$ and for nonzero $x\in\R^d$ define
	\begin{equation}
\label{e:normdef1}
		|x|_a = \frac{1}{m_a}\sum_{i=1}^d  x_i \arcsinh  (x_i u_a(x)).
	\end{equation}
\end{defn}

It follows from \eqref{e:m0def} and Taylor's theorem that, as $a \to 0$,
\begin{equation}
\label{e:maasy}
    m_a = \sqrt{2d}a(1+O(a^2)).
\end{equation}
Equation \eqref{e:normdef1} defines a norm on $\R^d$ whose properties are
indicated in Proposition~\ref{prop:norm}.  In particular, the norm $|\cdot|_a$ interpolates
between the $\ell_1$ norm when $a=\infty$ and the $\ell_2$ norm when $a=0$.
The norm's unit ball in dimensions $d=2,3$ is depicted in Figure~\ref{fig:norm}.

\begin{figure}[ht]
\centering{
\includegraphics[width=35mm, height=35mm]{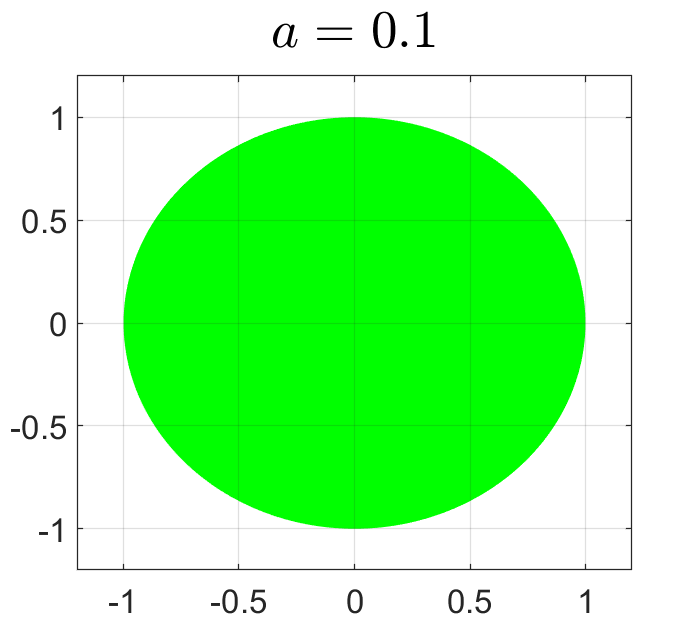}
\hspace{5mm}
\includegraphics[width=35mm, height=35mm]{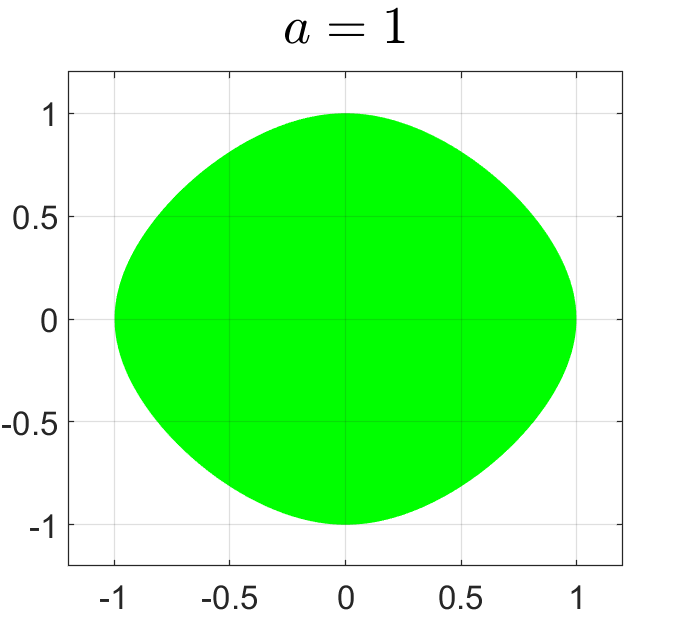}
\hspace{5mm}
\includegraphics[width=35mm, height=35mm]{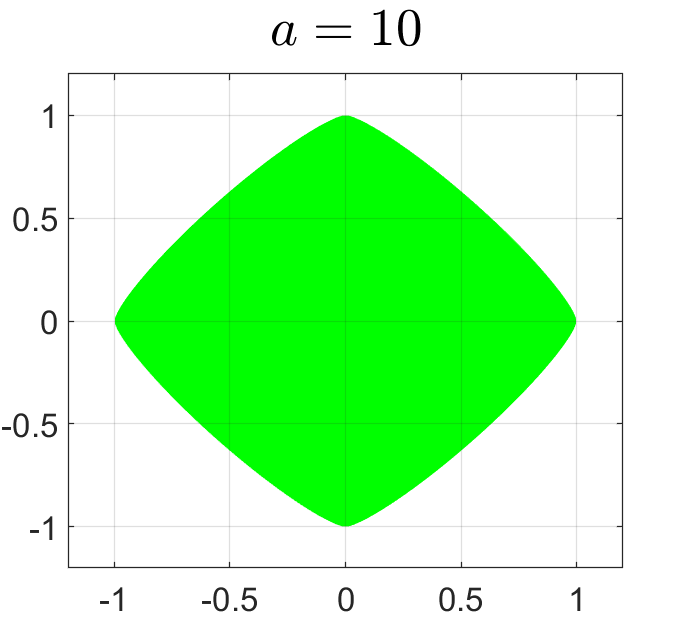}

\vspace{5mm}

\includegraphics[width=4cm, height=4cm]{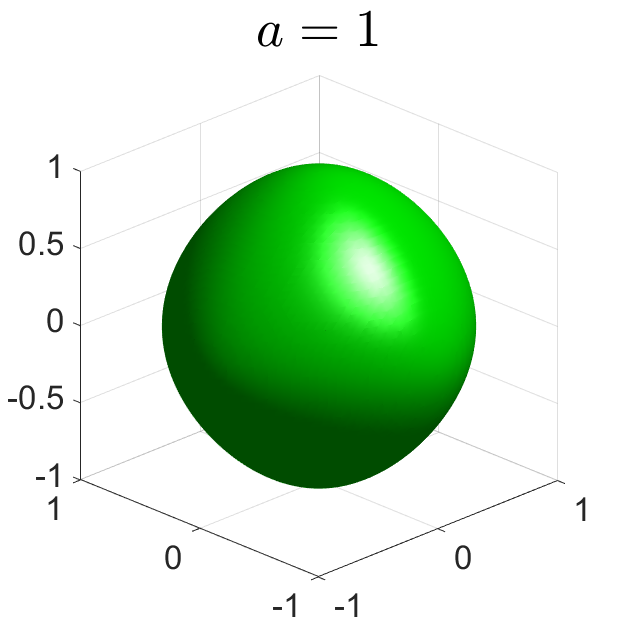}
\hspace{5mm}
\includegraphics[width=4cm, height=4cm]{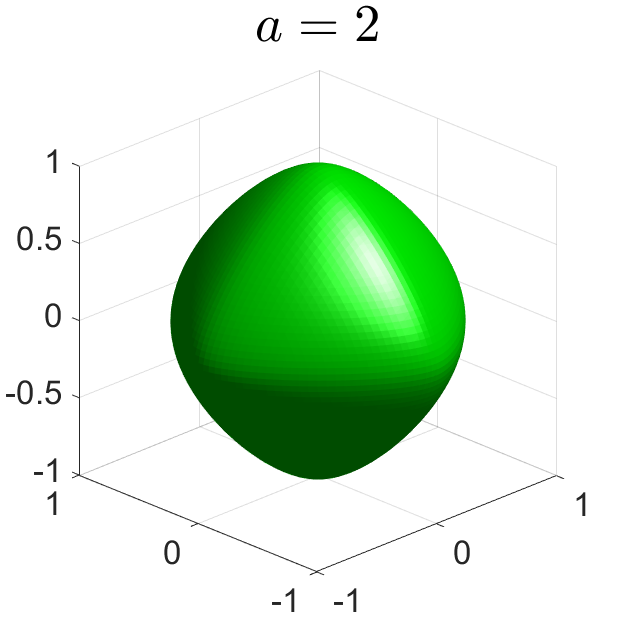}
\hspace{5mm}
\includegraphics[width=4cm, height=4cm]{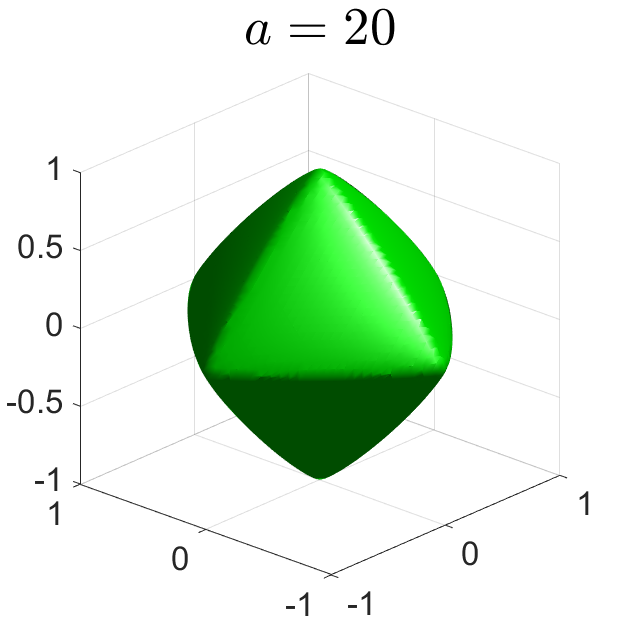}

\caption{Unit ball for the norm $|\cdot |_a$ in dimensions $d=2,3$.}
\label{fig:norm}}
\end{figure}

\begin{prop}
\label{prop:norm}
Let $d \ge 1$ and $a>0$.
The function $|\cdot|_a$ defines a norm on $\R^d$ which is monotone increasing
in $a$ and for all $x\in\R^d$ obeys
\begin{equation}
\lim_{a\to 0}|x|_a = \|x\|_2, \qquad \lim_{a\to \infty}|x|_a  = \|x\|_1,
\end{equation}
in fact $|x|_a = \|x\|_2(1+O(a^2))$ uniformly in $x \neq 0$.
In particular,
\begin{equation}
\label{e:normineq}
\|x\|_2 \le |x|_a \le \|x\|_1.
\end{equation}
\end{prop}

By definition, $m_a(0)=0$ and
$m_a$ is a strictly positive strictly increasing function of $a>0$.
To understand why the factor $m_a^{-1}$ appears on the right-hand side
of \eqref{e:normdef1}, we note that for any $a>0$ and for any unit vector $e_j\in\R^d$,
\begin{align}
\label{e:ue}
    u_a(e_j) &= \sqrt{(1+da^2)^2-1} = \sqrt{\cosh^2 m_a -1} = \sinh m_a,
\end{align}
and hence, for all $a > 0$,
\begin{equation}
\label{e:norm1}
    |e_j|_a = 1.
\end{equation}

The $|\cdot |_a$ norm originated in the analysis of the 2-dimensional Ising model
\cite[pp.~302--303]{MW73},
although it was not identified there
as a norm.  A proof that it defines a norm was given in \cite[Lemma~6.5]{Pfis99};
there the proof of the triangle inequality was based on the second Griffiths inequality
applied to the 2-dimensional Ising model.
We provide a simple alternate proof based on a random walk argument.
We do not know of any direct proof of the triangle inequality
based on the definition of the norm.
Neither are we aware of any prior proof of the monotonicity of the norm.

\subsection{The continuum Green function}

In Appendix~\ref{sec:continuumGreen} we consider and interpret the
integral
\begin{equation}
\label{e:Gx}
    G_\scale^{(\q)}(x)
    = \int_{\R^d} \frac{e^{ik\cdot x}}{(\frac{1}{2d}\|k\|_2^2+\scale^2)^{\q}} \frac{\D k}{(2\pi)^d}
    \qquad
    (x \in \R^d \setminus \{0\}),
\end{equation}
which in the case $\q=1$ is the Green function
for the (normalised) continuum Laplace operator on $\R^d$.
It follows from \eqref{e:Gx} that there is a scaling relation
\begin{equation}
\label{e:Gscaling}
    G_\scale^{(q)}(x) = \scale^{d-2\q} G_1^{(q)}(\scale x)
    \qquad
    (\scale>0).
\end{equation}
In Appendix~\ref{sec:continuumGreen}, we
recall the elementary proof that
for $\scale >0$, integers $d \ge 1$, and nonzero $\q\in \R$, the massive and massless
continuum Green functions are given explicitly
(in the sense of tempered distributions) by \eqref{e:Gscaling} together with
\begin{alignat}{2}
\label{e:Gaqx-2d}
    G_1^{(\q)}(x)
    & =
    \frac{2d^\q}{\Gamma(q)(2\pi)^{d/2}}
    \Bigg( \frac{\sqrt{2d} }{\|x\|_2} \Bigg)^{(d-2\q)/2}
    \!\!\!\!\!
    K_{(d-2\q)/2}(\sqrt{2d} \|x\|_2)
    ,
    \\
\label{e:G0qx}
    G_0^{(\q)}(x)
    & =
    \frac{d^\q \Gamma(\frac{d-2\q}{2})}{2^{\q}\pi^{d/2}\Gamma(\q)}
    \frac{1}{\|x\|_2^{d-2\q}},
\end{alignat}
where $K_\alpha$ is the modified Bessel
function of the second kind
and for \eqref{e:G0qx} we restrict to $d> 2\q$.
For $\alpha >0$ the asymptotic behaviour of $K_\alpha$ is known to be
\begin{align}
\label{e:Kasy}
    K_\alpha(z) &\sim \frac{\Gamma(\alpha)}{2} \Big(\frac{2}{z}\Big)^\alpha
    \quad (z \to 0),
    \qquad
    K_\alpha(z) \sim \sqrt{\frac{\pi}{2z}} e^{-z}
    \quad (z \to \infty).
\end{align}

\subsection{Asymptotic behaviour}
\label{sec:ab}

\subsubsection{Main result}

The following theorem gives a precise statement of the asymptotic decay of the lattice
Green function for arbitrary dimension $d \ge 1$ and for $a\ge 0$ (possibly $n$-dependent).
The norm $|\cdot|_a$ plays a key role in the anisotropic limit \eqref{e:mr}, for
which lattice effects persist when $a$ is independent of $n$.
Recall that $m_a=\arccosh (1+da^2)$ for $a \ge 0$.
We write $f(n) \sim g(n)$ to mean $\lim_{n\to\infty}f(n)/g(n)=1$.

\begin{theorem}
\label{thm:green-asy}
Let $d \ge 1$ and
$q\in (0,\infty)$ (not necessarily integer).  Fix $x \in \Z^d \setminus \{0\}$.
\\
{\bf (i) (Ornstein--Zernike decay).}
Let $a_n\in (0,\infty)$ and suppose that $a_nn \to \infty$ in such a manner that
$a_n$ remains bounded
(this includes in particular the case of fixed $a_n=a\in (0,\infty)$).
There exists $c_{a,\q,\hat x}>0$ depending on $a$, $\q$, and the direction
$\hat x = x/|x|_a$ (and on the dimension $d$), such that, as $n \to \infty$,
\begin{equation}
\label{e:mr}
    C_{a_n}^{(\q)}(nx)
    \sim
    c_{a_n,\q,\hat x}
    \frac{m_{a_n}^{(d-1-2\q)/2}}{(n|x|_{a_n})^{(d+1-2\q)/2} }
    \;
    e^{-m_{a_n}n|x|_{a_n}} .
\end{equation}
The ratio of the above left- and right-hand sides
converges to $1$ uniformly in nonzero $x$, and
the constant $c_{a,\q,\hat x}$ has the explicit $\hat x$-independent limit
 \begin{equation}
 \label{e:calim}
    c_{0,\q}
    =
    \lim_{a \to 0}c_{a,\q,\hat x}
    =
    \frac{d^q}{(2\pi)^{(d-1)/2}\Gamma(q)}.
\end{equation}
{\bf (ii) (Critical decay).}
Let
$a_n = \scale /n$ with $\scale \in [0,\infty)$, with $d>2\q$ if $\scale=0$.
Then, as $n \to \infty$,
\begin{equation}
\label{e:mr-critical}
    C_{a_n}^{(\q)}(nx)
    \sim
    \frac{1}{n^{d-2\q}} G_{\scale}^{(\q)}(x)
    .
\end{equation}
\end{theorem}

The asymptotic formula \eqref{e:mr} encompasses both regimes (I) and (II)
mentioned in Section~\ref{sec:intro}.
The anisotropic OZ regime (I) is the case of fixed $a_n=a>0$, for which the
anisotropic norm $|x|_a$ plays a role.

For the isotropic OZ regime (II), we
are interested in the case where $a_n\to 0$ in such a way that $a_nn\to\infty$
and $a_n^3n\to 0$.
Recall from \eqref{e:maasy}
and Proposition~\ref{prop:norm} that $m_a=\sqrt{2d}a(1+O(a^2))$ and
$|x|_a=\|x\|_2(1+O(a^2))$.  Consequently, as $a_n \to 0$ we have
\begin{equation}
    m_{a_n} n |x|_{a_n} = \sqrt{2d} a_n n \|x\|_2(1+O(a_n^2))
\end{equation}
and hence it follows from \eqref{e:mr} that if $a_nn \to\infty$ then
\begin{equation}
\label{e:mr-OZ-error}
 C_{a_n}^{(\q)}(nx)
    \sim
        c_{0,\q}
    \frac{(\sqrt{2d}a_n)^{(d-1-2\q)/2}}{(n\|x\|_{2})^{(d+1-2\q)/2}}
    e^{-\sqrt{2d} a_nn \, \|x\|_2[1+O(a_n^2)]}.
\end{equation}
If we now assume additionally that $a_n^3n\to 0$ then the error term in the
exponential can be neglected and we obtain the result claimed for regime (II), namely
\begin{equation}
\label{e:mr-OZ-a}
 C_{a_n}^{(\q)}(nx)
    \sim
        c_{0,\q}
    \frac{(\sqrt{2d}a_n)^{(d-1-2\q)/2}}{(n\|x\|_{2})^{(d+1-2\q)/2}}    e^{-\sqrt{2d} a_nn \, \|x\|_2}.	
\end{equation}
If the condition $a_n^3n\to 0$ is violated then we see from \eqref{e:mr-OZ-error} that
modifications to the exponential decay will occur from the error term in the exponent.

The massive critical regime (III) and the massless critical regime (IV)
are respectively the $\scale >0$ and $\scale =0$ cases of \eqref{e:mr-critical}.
There is coherence between regimes (II)
and (III) in the sense that if $a_n$ in (II) is replaced by $\scale/n$ then the exponential
factor becomes a constant and the powers $a_n^{(d-1-2\q)/2} n^{-(d+1-2\q)/2}$ reduce
to an $\scale$-dependent multiple of $n^{-(d-2\q)}$.  The continuum Green function
$G_\scale^{(q)}(x)$ provides the amplitude for the asymptotic decay in the critical regimes.
There is no statement of uniformity in $x$ in \eqref{e:mr-critical} because uniformity is
impossible for $\scale>0$:
e.g., if $x=n^2y$ with $y$ independent of $n$ then $\frac{\scale}{n} x
=\scale ny \to \infty$ as $n \to \infty$ and we are actually in regime (I), not regime (III).

\subsubsection{Previous results}

The proof of Theorem~\ref{thm:green-asy} is based on the representation
\begin{align}
\label{e:BesselC}
    C^{(\q)}_{a}(x)
				 &=
 \frac{1}{\Gamma(\q)}\int_0^\infty t^{\q-1}e^{-(a^2+1)t}	
    \prod_{j=1}^{d}I_{x_j}(t/d) \D t
\end{align}
in terms of the modified Bessel function of the first kind.
Much of Theorem~\ref{thm:green-asy} has been proved previously by other authors,
and we now describe what was done previously and how our approach simplifies, extends and
unifies earlier work.

For $\q=1$ and for fixed $a_n=a>0$, the asymptotic formula \eqref{e:mr} is proved
in \cite[Theorem~3.2]{MY12} for $d \ge 1$, and for $d=2$ in
\cite[Proposition~13]{Mess06}.
Neither of those references identified the role of the anisotropic norm in
\eqref{e:mr}, and the norm
makes the statement significantly more transparent.
In \cite[Theorem~3.3]{MY12},
\eqref{e:mr-OZ-a} is stated to hold in
the limit in which $a_n \to 0$ with $a_nn\to \infty$; in fact
the further restriction $a_n^3n\to 0$ is necessary for the
simplification of the exponential in \eqref{e:mr} to yield the isotropic form \eqref{e:mr-OZ-a}.
Our method of proof is based on the Laplace method as in
\cite{MY12} but it is simplified by our appeal to well-established properties of the modified
Bessel function rather than deriving them as part of the proof as in \cite{MY12}.
Also, unlike the separate proofs for the anisotropic and isotropic cases in \cite{MY12},
we give one unified proof.

The massive critical regime was considered in \cite{PS99}
(indeed these authors computed higher-order terms as well), but the arguments used in \cite{PS99} do not constitute a proof.
The formula \eqref{e:mr-critical} for $\q=1$ and
$\scale > 0$ can be inferred from
the statement of \cite[Proposition~3.1]{DGGZ22}, which is proved via
the local central limit theorem.
Our proof, which again uses known properties of the modified Bessel function,
involves a straightforward application of the dominated convergence theorem and does
not involve the Laplace method.

For the massless critical regime (IV),
the asymptotic behaviour of the
critical lattice polyharmonic Green function is given in \cite{Mang67} as
\begin{equation}
\label{e:Cq0}
    C^{(\q)}_{0}(x) \sim
    \frac{d^q\Gamma(\frac{d-2\q}{2})}{2^{\q}\pi^{d/2}\Gamma(\q)} \frac{1}{\|x\|_2^{d-2\q}}
    \qquad
    (\q=1,2,3,\ldots; \; d>2\q)
\end{equation}
with explicit higher-order correction term.  Since higher-order terms
are known we make no effort
here to compute them, as our focus in the proof is on simplicity.  We prove \eqref{e:Cq0}
as the $\scale=0$ case of \eqref{e:mr-critical} for arbitrary
real $\q\in (0,\infty)$ when $d>2\q$.  This special case of
our proof of \eqref{e:mr-critical} in the entire critical regime $\scale \ge 0$
does not require separate attention.
When $\q=1$, \eqref{e:Cq0}
gives the well-known decay of
the critical lattice Green function.
In fact, the $\|x\|_2^{-(d-2)}$ decay in  \eqref{e:Cq0} for $\q=1$
holds more generally
under a second-moment condition for $D_{xy}$ (recall \eqref{e:DeltaJ}), with error term
of order $\|x\|_2^{-d}$ with known coefficient.  For $\q=1$ see, e.g.,
\cite[p.~82]{LL10} or \cite[Theorem~3.4]{MY12}, or \cite[p.~308]{Spit76} for $d=3$, and for
further error terms see
\cite{Uchi98}.  A version of \eqref{e:Cq0} for $\q=1$ holds under certain conditions even when
the transition matrix $D_{xy}$ is permitted to assume negative values  \cite{Hara08}.

\subsubsection{Explicit calculation for $d=1$}

For $d=1$ and integers $\q\ge 1$, the
condition $d=1>2\q$ is violated and $C_0^{(\q)}=\infty$, so regime (IV) does not apply.
The computation of $C_a^{(q)}(x)$ for $d=1$, integer $\q\ge 1$, and $a>0$ can be
done explicitly
with the result that
\begin{align}
    \int_{-\pi}^\pi \frac{e^{ikx}}{(1+a^2-\cos k)^\q} \frac{dk}{2\pi}
    & =
    \frac{e^{-m_a|x|}}{\sinh^\q m_a}
    \sum_{l=0}^{q-1}
     \binom{|x|+q-1}{q-1-l} \binom{q-1+l}{l}
     \Big( \frac{e^{-m_a}}  {2\sinh m_a}
       \Big)^l
\label{e:d1generalq}
\end{align}
with $m_a = \arccosh (1+a^2)$.
The above formula can be verified by residue calculus or by
an appropriate rewriting of the formula \cite[(3.616.7)]{GR07}.
In detail, the cases $\q=1$ and $\q=2$ are
\begin{align}
\label{e:d1q1}
    C_a^{(1)}(x)
    &=
    \frac{e^{-m_a|x|}}{\sinh m_a}
    \qquad\qquad\qquad\qquad\qquad\qquad\qquad\;\;
    (d=1),
\\
\label{e:d1q2}
    C_a^{(2)}(x)
    &=
    \frac{|x| e^{-m_a|x|}}{\sinh^2 m_a}
    \Bigg[ 1 + \frac{1}{|x|} \Big(1 +  \frac{e^{-m_a}}  {\sinh m_a} \Big) \Bigg]
    \qquad\quad
    (d=1)   .
\end{align}
Both of the formulas \eqref{e:d1q1}--\eqref{e:d1q2} refine and are consistent with
\eqref{e:mr} and \eqref{e:mr-critical} from Theorem~\ref{thm:green-asy}.
In particular, for $a=\scale/n$ with fixed $\scale>0$, \eqref{e:d1q1} gives
\begin{equation}
    C_{\scale/n}^{(1)}(nx) \sim \frac{n}{\sqrt{2}\scale}e^{-\sqrt{2} \scale |x|},
\end{equation}
and since $K_{-1/2}(y) = K_{1/2}(y) = \pi^{1/2}(2y)^{-1/2}e^{-y}$ for $y >0$ (see \cite[8.432.8]{GR07})
this agrees with \eqref{e:mr-critical}.

\subsubsection{Ornstein--Zernike vs critical decay}

In the physics literature, the inverse mass $\xi_a=m_a^{-1}$ is known as the
\emph{correlation length}.
With $a=\scale/n$ and $\scale >0$,
Theorem~\ref{thm:green-asy} can then be interpreted informally as identifying the
following decay of the lattice Green function:
\begin{alignat}{3}
    &\scale >0   &\qquad  n\|x\|_2 \asymp \xi_a
    &\qquad  \text{massive continuum limit},
    \\
    & \scale \to \infty &\qquad  n\|x\|_2 \gg \xi_a
    &\qquad  \text{Ornstein--Zernike decay}.
\end{alignat}
For the latter case, we see the Euclidean invariance if $\scale=o(n)$ but not for
$\scale = an$
with fixed $a>0$.

The Ornstein--Zernike and critical regimes occur in general dimensions
in lattice statistical mechanical
models such as the self-avoiding walk, percolation, and the Ising model
\cite{CIV08,CC86b,CCC91}.
This perspective is standard in the physics literature but a mathematical description
of the near-critical behaviour which crosses over between the two regimes
is lacking in most examples, even in high dimensions where the lace expansion applies.
The asymptotic formula \eqref{e:mr} provides a prototype for what can be expected
for the near-critical two-point functions of
the high-dimensional statistical mechanical models.

The bounds in regimes (I)--(II) in general do \emph{not} hold uniformly in all
$a > 0$, $n \ge 1$, and nonzero $x\in\Z^d$.  This is evident in the
explicit formula \eqref{e:d1q2} for $d=1$ and $q=2$,
where
the first term $|x| \frac{e^{-m_a|x|}}{\sinh^2 m_a}$
dominates when $x \to \infty$ with fixed $a$, in agreement with \eqref{e:mr},
whereas with fixed $x$, in the limit $a\to 0$
we have $m_a\sim \sqrt{2}a \to 0$, the exponentials become insignificant, and \eqref{e:d1q2}
is dominated by the factor $\sinh^{-3}m_a \sim (\sqrt{2}a)^{-3}$
arising from  its last term.
This shows the impossibility for this case
of an upper bound of the form $|x|a^{-2} e^{-m_a|x|}$
that is uniform in both $x$ and $a$.

Similarly,  for $d>3$ and $\q=1$, there can be no upper bound
on $C_a^{(1)}$ of the form
\begin{equation}
\label{e:OZnotuniform}
    m_a^{(d-3)/2}\frac{1}{|x|_a^{(d-1)/2}}e^{-m_a|x|_a}
\end{equation}
that is uniform in all $a>0$ and nonzero $x \in \Z^d$,
because \eqref{e:OZnotuniform} vanishes as $ a\to 0$ with fixed $x$
due to the factor $m_a^{(d-3)/2}$, whereas if $|x|_a$ grows like $m_a^{-1}$ then
$C_a^{(1)}$ is in regime (III) and decays as a multiple of $\|x\|_2^{-(d-2)}$.

It remains an open problem to determine for which values of $d,\q$ the formula
\eqref{e:mr} in fact gives a bound which is unform in $a >0$ and nonzero $x$.
On the other hand,
for $\q=1$ and $d>2$  an upper bound that \emph{is} uniform in $a \ge 0$ and in $x$
is given in \cite[Proposition~2.1]{Slad20_wsaw}, which asserts that
there are constants $\kappa_1>0$ and $\kappa\in (0,1)$ such that
for all $a \ge 0$ and all $x \neq 0$,
\begin{equation}
\label{e:srwbd}
    C^{(1)}_a(x) \le \kappa_1 \frac{1}{|x|_a^{d-2}}e^{-\kappa m_a|x|_a}
    .
\end{equation}
(By changing the constants, another norm than $|x|_a$ could be used in the above.)
As in \cite[Lemma~3.3]{MS22}, the inequality \eqref{e:srwbd} easily implies that
for general integers $\q\ge 1$ and dimensions $d>2\q$,
\begin{equation}
\label{e:srwbd-q}
    C_a^{(\q)}(x) \le \kappa_\q \frac{1}{|x|_a^{d-2\q}}e^{-\kappa m_a|x|_a}
    .
\end{equation}
The uniform upper bound \eqref{e:srwbd} combines
the critical $|x|_a^{-(d-2)}$ decay with the exponential decay for $a>0$.
The relaxation of the exponential decay via $\kappa<1$ compensates for the
differing power laws in \eqref{e:srwbd-q} and in regime (I).
Bounds of the form \eqref{e:srwbd} have been proved and applied to analyse the
critical behaviour of weakly self-avoiding
walk in dimensions $d>4$ \cite{MS22,Slad20_wsaw} and of percolation in dimensions
$d>6$ \cite{HMS22}.

\subsection{Organisation}

The remainder of the paper is organised as follows.

In Section~\ref{sec:Bessel}, we give the elementary derivation of the
representation \eqref{e:BesselC} of $C_a^{(q)}(x)$ in terms of the modified
Bessel function $I_\nu$.  This representation in terms of a $1$-dimensional integral
is the basis for all of our analysis.  We then recall properties of $I_\nu$ which
enable the asymptotic evaluation of the integral \eqref{e:BesselC}.

In Section~\ref{sec:subcrit}, we prove Theorem~\ref{thm:green-asy}(i), pertaining
to the Ornstein--Zernike regime.  In this regime,
the Bessel integral \eqref{e:BesselC} has an exponential
factor in the integrand which makes it amenable to application of the Laplace method.
The norm $|\cdot|_a$ emerges naturally from a computation involving
the critical point which dominates the behaviour arising in the Laplace method.

In Section~\ref{sec:contlim}, we prove Theorem~\ref{thm:green-asy}(ii), pertaining
to the critical regime.
In the critical regime, there is no longer any exponential behaviour in the integrand
of the Bessel integral \eqref{e:BesselC} and there is no need for the Laplace method.
Given the well-known asymptotics for $I_\nu$ recalled in Section~\ref{sec:Bessel},
the proof follows quickly from the dominated convergence theorem.

Finally, in Appendix~\ref{sec:continuumGreen} we provide an elementary proof that
the formulas \eqref{e:Gaqx-2d}--\eqref{e:G0qx} for the continuum Green function are
equal to the integral \eqref{e:Gx} over $\R^d$ in the sense of tempered distributions,
and in Appendix~\ref{sec:Besselpf} we discuss properties of the modified Bessel function.

\section{Bessel representation}
\label{sec:Bessel}

For any integer $\nu \geq 0$ and $ t \in \R$ the modified Bessel function of the first kind $I_\nu(t)$ is given by
\begin{equation}
	I_\nu(t) = \frac{1}{\pi}\int_0^\pi e^{t \cos\theta+i\nu\theta}\D \theta.
\end{equation}
For our purposes it is more useful to consider
\begin{equation}
\label{e:Ibardef}
	\bar I_\nu(t)
    =
    e^{-t} I_\nu(t)
    =
    \frac{1}{\pi}\int_0^\pi e^{-t (1-\cos\theta)+i\nu\theta}\D \theta
\end{equation}
which has the exponential growth of $I_\nu(t)$ cancelled.
The following lemma provides the well-known integral representation that
is the foundation for the proof of Theorem~\ref{thm:green-asy}.

\begin{lemma}
\label{lem:CBessel}
For $d \ge 1$, $a \ge 0$, $\q>0$,  $x \in \Z^d$, and with the restriction $d>2\q$ when $a=0$,
\begin{align}
\label{e:Besselrep0}
    C^{(\q)}_{a}(x)
				 &=
 \frac{1}{\Gamma(\q)}\int_0^\infty t^{\q-1}e^{-a^2t}	
    \prod_{j=1}^{d}\bar I_{x_j}(t/d) \D t.
\end{align}
\end{lemma}
\begin{proof}
Let $\hat F(k) =a^2+1-\hat D(k)$.
We use the identity
\begin{align}
\label{e:reciprocal}
    \frac{1}{v^{\q}}  &= \frac{1}{\Gamma(\q)}\int_0^\infty t^{\q-1}e^{-tv}\D t
    \qquad
    (v>0)
\end{align}
in the definition \eqref{e:Cxq} to obtain
\begin{align}
\label{e:Besselrep}
    C^{(\q)}_{a}(x)
	&= \int_{[-\pi,\pi]^d}\frac{e^{ik\cdot x}}{\hat F(k)^\q}\frac{\D k}{(2\pi)^d}
    \nnb
      &= \frac{1}{\Gamma(\q)}\int_{[-\pi,\pi]^d}
      \int_0^\infty t^{\q-1}e^{- t\hat F(k)}\D t\;
      e^{ikx} \frac{\D k}{(2\pi)^d}\nnb
	&= \frac{1}{\Gamma(\q)}\int_0^\infty t^{\q-1}e^{-a^2t}	
    \prod_{j=1}^{d}\bar I_{x_j}(t/d) \D t,
\end{align}
and the proof is complete.  Note that there is no issue with convergence of this last
integral at $t=0$, and for large $t$ convergence is guaranteed
(assuming $d>2\q$ when $a=0$)
by the fact that
$\bar I_\nu(z) \sim (2\pi z)^{-1/2}$ as $z \to \infty$.   In particular, this justifies
the above application of Fubini's Theorem.
\end{proof}

To study the Ornstein--Zernike regime we apply the change of variable $t=dnv$
to the integral representation \eqref{e:Besselrep0} to obtain
\begin{align}
\label{e:BesselOZ}
    C^{(\q)}_{a}(nx)
				 &=
 \frac{d^{q}n^q}{\Gamma(\q)}\int_0^\infty v^{\q-1}e^{-dna^2v}
    \prod_{j=1}^{d} \bar I_{nx_j}(nv) \D v
    .
\end{align}
For the continuum regime we will also make the replacement $v=nt/d$ in \eqref{e:BesselOZ}
and use
\begin{align}
    C_{{s}/n}^{(\q)}(nx)
    &=
    \frac{n^{2\q}}{\Gamma(\q)} \int_0^\infty  t^{\q-1}
    e^{-\scale^2 t} \prod_{j=1}^{d}\bar I_{n x_j} (n^2t/d) \, \D t
    .
\label{e:Besselcont}
\end{align}

To
study the integrals \eqref{e:BesselOZ}--\eqref{e:Besselcont}
we will make use of well-established asymptotic
properties for $I_\nu$.
To state these properties, for $\nu > 0$ and $t >0$ we define
\begin{equation}
\label{e:L_psi_def}
	L_\nu(t) = \frac{1}{(2\pi\nu)^{1/2}} \frac{e^{\nu \psi(t)}}{(1+t^2)^{1/4}},
    \qquad
    \psi(t) =
    -t+
    \sqrt{1+t^2} + \log\Big( \frac{t}{1+\sqrt{1+t^2}}\Big).
\end{equation}
The identity $\log( \frac{t}{1+\sqrt{1+t^2}}) = -\arcsinh(t^{-1})$ gives a
useful alternate representation for $\psi$.
The first three derivatives of $\psi$ are:
\begin{align}
\label{e:psi_deriv}
	\psi'(t)
	&= -1+ \sqrt{1+t^{-2}},\\
\label{e:psi_deriv2}
	\psi''(t)
	&= -\frac{t^{-3}}{\sqrt{1+t^{-2}}},\\
\label{e:psi_deriv3}
	\psi'''(t)
	&= \frac{2t^{-6}+3t^{-4}}{(1+t^{-2})^{3/2}}.
\end{align}
The following lemma
gives asymptotic representations
of the Bessel function of large argument and large order.
The proof of the lemma is deferred to Appendix~\ref{sec:Besselpf}.
We use \eqref{e:ILasy} for the OZ regime and \eqref{e:Inunu}--\eqref{e:Lclaim}
for the continuum regime.

\begin{lemma}
\label{lem:Inunu2}
As $\nu \to \infty$,
\begin{equation}
\label{e:ILasy}
	\bar I_\nu(\nu t) = L_\nu (t)(1+o(1))
\end{equation}
where the $o(1)$ is uniform in $t>0$.
Also, as $\nu \to \infty$, for any $s >0$,
\begin{align}
\label{e:Inunu}
    \bar I_\nu(\nu^2 s) &\sim \frac{e^{-1/2s}}{\nu (2\pi s)^{1/2}},
\end{align}
with an error that is not uniform in $s$.
Finally, there exist $C,\delta, \nu_0 >0$ such that
\begin{align}
\label{e:Lclaim}
    \bar I_{\nu}(\nu^2 s)
    &\le C
    \Big(
    e^{ - \delta\nu} \1_{2\nu s < 1}
    +
    \nu^{-1} s^{-1/2} e^{- \delta/s} \1_{2\nu s \ge 1}
    \Big)
    \qquad
    (\nu \ge \nu_0,\; s>0).
\end{align}
\end{lemma}

\section{Ornstein--Zernike regime: Proof of Theorem~\ref{thm:green-asy}(i)}
\label{sec:subcrit}

In this section, we prove Theorem~\ref{thm:green-asy}(i).
Let $x$ be a vector in $\Z^d \setminus \{0\}$, and without loss of generality assume that
$x_1 \ge x_2 \ge \cdots \ge x_d \ge 0$.  We write $\xtop$ for the number of nonzero components of $x$.  Throughout this section, we consider a bounded sequence $a_n \in (0, a_{\max}]$ with $a_nn\to\infty$.
To lighten the notation, we write simply $a$ in place of $a_n$.
In particular, $a$ can be independent of $n \to \infty$, or we can have $a \to 0$ as long
as $an \to \infty$.

We start with \eqref{e:BesselOZ}, which states that
\begin{align}
    C^{(\q)}_{a}(nx)
				 &=
 \frac{d^{q}n^q}{\Gamma(\q)}\int_0^\infty v^{\q-1}e^{-dna^2v}
 (\bar I_0(nv))^{d-\xtop}
    \prod_{j=1}^{\xtop} \bar I_{nx_j}(nv) \D v
    .
\end{align}
With the asymptotic formula for $\bar I_\nu(\nu t)$ from \eqref{e:ILasy} together
with the definitions of $L_\nu$ and $\psi$
from \eqref{e:L_psi_def}, after some algebra this leads to
\begin{align}
\label{e:CqOZI}
    C^{(\q)}_{a}(nx)
     &= (1+\delta_n)
     \alpha_\q n^{\q-\xtop/2}
     \int_{0}^\infty h_{n,x}(v)
     e^{-ng_{a,x}(v)}  	
     \D v
     ,
\end{align}
where $\delta_n \to 0$ (uniformly in nonzero $x$) and
\begin{align}
\label{e:alphaqdef}
    \alpha_\q &= \frac{d^\q}{(2\pi)^{d/2}\Gamma(\q)}, \\
\label{e:hdef}
	h_{n,x}(v) &=  v^{\q-1}
    (\sqrt{2\pi}\bar I_0(nv))^{d-\xtop}
    \prod_{j=1}^{\xtop} \frac{1}{(x_j^2+v^2)^{1/4}}, \\
\label{e:gdef}
	g_{a,x}(v) &=
    da^2v - \sum_{j = 1}^\xtop  x_j \psi(v/x_j)
.
\end{align}

We first solve $g_{a,x}'(v) = 0$.  By definition of $g_{a,x}$ and by \eqref{e:psi_deriv},
\begin{align}
\label{e:gder}
	 g_{a,x}'(v) &= d(1+a^2) - \sum_{j=1}^d \sqrt{1+v^{-2}x^2_j} .
\end{align}
By the definition of $u_a(x)$ in \eqref{e:udef1}, we see that
the unique solution of the equation $g_{a,x}'(v) = 0$
is $v_a(x)=u_a^{-1}$, where for notational convenience we write $u_a(x)$ simply as $u_a$.
  We will soon see that this solution is
the location of the unique minimum of $g_{a,x}$.
Since we are allowing the variable $a$ to go to zero,
which sends $v_a(x)$ to infinity, it is convenient to relocate this minimum to $1$.
We therefore rescale
the representation \eqref{e:CqOZI} via $v = y/u_a$ and obtain
\begin{align}
\label{e:CqOZII}
    C^{(\q)}_{a}(nx)
     &= (1+\delta_n)
     \alpha_\q \Big(\frac{n}{u_a}\Big)^{\q-d/2}
     \int_0^\infty \bar h_{n,a,x}(y)
     e^{-n\bar g_{a,x}(y)}\D y
\end{align}
with
\begin{align}
\label{e:bardefs}
	\bar h_{n,a,x}(y) =
	h_{n}(y/u_a)n^{(d-r)/2}u_a^{q-1-d/2}, \quad
	\bar g_{a,x}(y) =
    g_{a,x}(y/u_a)
    .
\end{align}
The minimum of $\bar{g}_{a,x}$ is located exactly at $1$, as
is illustrated in Figure~\ref{fig:g}.

\begin{figure}
\center
\includegraphics[scale = 0.5]{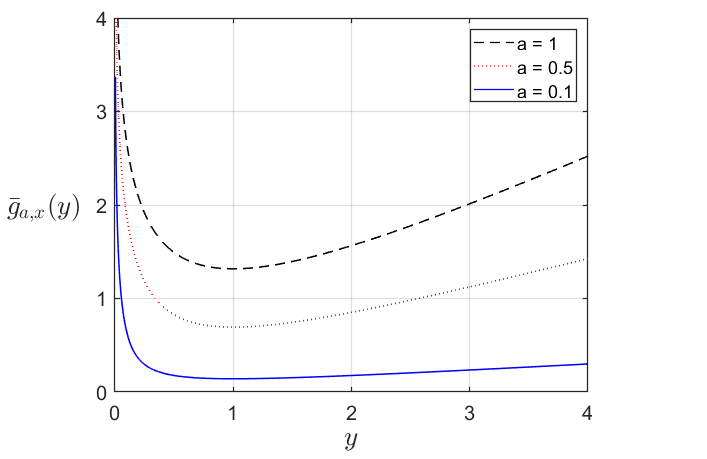}
\caption{Plot of $\bar g_{a,x}$ for different values of $a$, for $d=1$ and $x=1$.
\label{fig:g}
}
\end{figure}

Recall the norm $|x|_a$ from Definition~\ref{def:norm}.
We write $\hat x = x/|x|_a$ and $\hat u_a = u_a(\hat x)$.
The scaling relation $\lambda u_a(\lambda x) = u_a(x)$ for all  $\lambda> 0$
follows from the definition of $u_a(x)$ in \eqref{e:udef1},
and implies that $u_a  x_j = \hat u_a \hat x_j$ and
$\hat u_a = |x|_a u_a$.
The definitions lead to
\begin{align}
\label{e:gbarxhat}
    \bar g_{a,x}(y) & = |x|_a \, \bar g_{a,\hat x}(y),
    \\
\label{e:bar_h_new}
    \bar h_{n,a,x}(y) & =
    y^{\q-1}
    \Big(\sqrt{2\pi n/u_a}\bar I_0(ny/u_a) \Big)^{d-\xtop}
    \prod_{j=1}^\xtop \frac{1}{(y^2+ \hat u_a^2 \hat x_j^2)^{1/4}}
    .
\end{align}
The $a$-dependence of $\bar{g}_{a,x}''(y)$
hinders an immediate application of a standard theorem for the Laplace method
such as \cite[Theorem~7.1, p.~127]{Olve97},
so we prove Theorem~\ref{thm:green-asy}(i) by analysing the integral in \eqref{e:CqOZII} directly.
To do so, we require the detailed understanding of
the $\bar g_{a,x}$ that we present next.
As a preliminary, we note that
it follows from the definition of $u_a$
in \eqref{e:udef1} that $\hat u_a  = O(a)$ as $a\to 0$ uniformly in $x \neq 0$, and moreover
that
\begin{align}
	\label{e:asymptotics_for_u}
	 \hat  u_a  = \sqrt{2d}  \frac{a}{\|\hat x\|_2}
    (1+O(a^2 ))
    \quad
    \text{uniformly in } x \neq 0 .
\end{align}

\begin{lemma}
\label{lem:ga}
Let $a>0$.
The function $\bar g_{a,x}$ is convex and attains its unique minimum on $(0 ,\infty)$ at $1$,
with $\bar g_{a,x}(1)= m_a|x|_a$.  Also
\begin{align}
\label{e:gbarprimeprime}
    \bar g_{a,x}''(y) &
    =
    | x|_a  \, \bar g_{a,\hat x}''(y)
     =
    | x|_a  \hat u_a
    \sum_{j=1}^d \frac{\hat x_j^2y^{-3}}{\sqrt{1+\hat u_a^2 \hat x_j^2 y^{-2}}}
     ,
    \\
\label{e:gbarppp}
    \bar g'''_{a,x}(y)
    &
    =
    | x|_a  \, \bar g_{a,\hat x}'''(y)
    = -
    |x|_a \hat u_a
    \sum_{j=1}^d
    \frac{3\hat x_j^2y^{-4}+2\hat u_a^2 \hat x_j^4y^{-6}}
    {(1+\hat u_a^2 \hat x_j^2 y^{-2})^{3/2}}
    .
\end{align}
In addition,
for any $\alpha \in \R$ and any $n \ge 1$,
\begin{align}
\label{e:g_at_infty}
	\lim_{y \to \infty} y^{\alpha} e^{-n \bar g_{a,x}(y)} = 0.
\end{align}
Finally, if $\alpha<0$ then
the maximum of $y \mapsto y^{\alpha} e^{-n\bar g_{a,x}(y)}$ for $y\in (0,1]$
is uniquely attained and lies in the interval $[\frac 12,1]$ provided that
$a$ is bounded and $an$ is
sufficiently large (depending on $\alpha$ but not on nonzero $x$).
\end{lemma}

\begin{proof}
By definition, $\bar g_{a,x}(1)=g_{a,x}(1/u_a)$ and
\begin{align}
	g_{a,x}(1/u_a) &=
    u_a^{-1} \Big[d(a^2+1)- \sum_{j = 1}^d \sqrt{1+u_a^{2}x_j^{2}} \,\Big]
    + \sum_{j = 1}^d x_j\arcsinh(u_ax_j) \nnb
    &=\sum_{j = 1}^d x_j\arcsinh(u_ax_j) = m_a|x|_a,
\end{align}
with the last equality due to the definition of the norm in \eqref{e:normdef1}.
For the derivatives we use \eqref{e:gbarxhat} together with \eqref{e:gdef} and
the expressions
for $\psi''$ and $\psi'''$ given in \eqref{e:psi_deriv2}--\eqref{e:psi_deriv3}.
In particular,
it follows from \eqref{e:gbarprimeprime} that
$\bar g_{a,x}$ is convex and therefore the unique critical point at $1$ is
the location of the unique minimum.

For \eqref{e:g_at_infty},
it suffices to consider $g_{a,x}$ since there is no claim of
uniformity in $a$.
It can be seen from the definition of $\psi$ in \eqref{e:L_psi_def} that
$\psi(t) \to 0$ as $t \to \infty$.
With the definition of $g_{a,x}$ in \eqref{e:gdef}, this implies that
$g_{a,x}(y) \sim da^2y$ as $y \to \infty$,
so \eqref{e:g_at_infty} holds for any $n \geq 1$.

Finally, and most substantially, we let $\alpha<0$ and
prove that the maximum of $y \mapsto y^{\alpha} e^{-n\bar g_{a,x}(y)}$ for $y\in [0,1]$
is uniquely attained and lies in the interval $[\frac 12,1]$, provided that $a=a_n$ is
bounded and $an$ is
sufficiently large (depending on $\alpha$ but not on nonzero $x$).
We write
\begin{equation}
	y^{-|\alpha|} e^{-n\bar g_{a,x}(y)} = \exp[-n \varphi_{a,x,\alpha}(y)]
\quad \text{with} \quad
\varphi_{a,x,\alpha,n}(y) = \bar g_{a,x}(y) + \frac{|\alpha|}{n}\log y.
\end{equation}
To find a critical point of $\varphi=\varphi_{a,x,\alpha,n}(y)$ we first observe, as in \eqref{e:gder}, that
\begin{align}
	\varphi'(y)
	= \frac{1}{u_a}F(y/u_a)
    \quad \text{with} \quad
\label{e:tilde_g_prime}
	F(t) = d(1+a^2) - \sum_{j=1}^d \sqrt{1+t^{-2}x^2_j}  + \frac{|\alpha|}{n}t^{-1}.
\end{align}
Note that
$F(u_a^{-1})=|\alpha|u_a/n>0$, and that  $F(t) \sim -t^{-1}(\|x\|_1-|\alpha|/n)$
as $t\to 0$ so $F(t)\to -\infty$ uniformly in $n \geq |\alpha|+1$ and $x \neq 0$.
To prove that $\varphi$ has a unique critical point in $[0,1]$, it therefore suffices to
prove that $F(t)$ is increasing on $t\in [0,u_a^{-1}]$.
The derivative of $F$ is
\begin{align}
	F'(t)
	&= t^{-2}G(t)
\quad \text{with} \quad
	G(t)=\sum_{j=1}^d \frac{t^{-1}x_j^2}{\sqrt{1+t^{-2}x_j^2}}-\frac{|\alpha|}{n}.
\end{align}
By multiplying by $t$ in the numerator and denominator of the above sum, we see that
$G$ is decreasing.
As $t \to 0$, $G(t) \to \|x\|_1- |\alpha|/n \ge 1- |\alpha|/n>0$ uniformly in $x\neq0$ and in $n \geq |\alpha|+1$.  Also, since $u_ax_j=\hat u_a \hat x_j$,
\begin{align}
	G(1/u_a)
&
= \sum_{j=1}^d \frac{\hat u_a\hat x_j x_j}{\sqrt{1+\hat u_a^2 \hat x_j^2}}- \frac{|\alpha|}{n}.
\end{align}
Recall \eqref{e:asymptotics_for_u} and \eqref{e:normineq}.
The square root on the right-hand side is bounded above since $a$ is bounded, so
there is a constant $c_0>0$ such that, uniformly in nonzero $x$,
\begin{align}
	G(1/u_a)
	&\geq c_0 \hat u_a \frac{\|x\|_2^2}{|x|_a} - \frac{|\alpha|}{n}
    \ge \hat u_a\Big(\frac{c_0\|x\|_2^2}{\|x\|_1}-\frac{|\alpha|}{n  \hat u_a}\Big) .
\end{align}
This proves that $G(1/u_a)>0$ for $an $ large enough (independent of $x \neq 0$).
Therefore $G(t)>0$ for all $t \in [0, u_a^{-1}]$, which
completes the proof that $F$ is increasing on $[0, u_a^{-1}]$.
As noted previously,
this proves that there exists a unique $t^*(a,n,x)\in [0,u_a^{-1}]$ such that $F(t^*) = 0$.

To conclude, we now verify that $t^*\in[(2u_a)^{-1}, u_a^{-1}]$.
It suffices to show that $F(1/(2u_a))<0$ if $an$ is large
enough (independent of $x\neq 0$).
By definition of $u_a$,
\begin{align}
	F(1/(2u_a))
	&= d(1+a^2)-\sum_{j=1}^d\sqrt{1+4u_a^2x_j^2} + 2u_a \frac{|\alpha|}{n} \nnb
	&=-\sum_{j=1}^d \Big( \sqrt{1+4\hat u_a^2\hat x_j^2}- \sqrt{1+\hat u_a^2\hat x_j^2}\Big)
    + 2\hat u_a \frac{|\alpha|}{n|x|_a}.
\end{align}
If $a$ is bounded below away from zero then the last term on the right-hand side is as
small as desired by taking $n$ large, whereas the difference in the first term  is bounded
below by a positive constant, so $F(1/(2u_a))<0$ in this case. It therefore suffices to
consider small $a$, for which we see that
\begin{align}
	F(1/(2u_a))
	&\le - c\sum_{j=1}^d  \hat u_a^2\hat x_j^2
    + 2\hat u_a \frac{|\alpha|}{n|x|_a}
    \le -c' \hat u_a^2
    + 2\hat u_a \frac{|\alpha|}{n|x|_a}
    \le - \hat u_a^2 \big(c' - \frac{2|\alpha|}{\hat u_a n } \big) .
\end{align}
For $an$ sufficiently large (independent of  $x\neq 0$)
we conclude that $F(1/(2u_a))<0$.
This completes the proof.
\end{proof}

Next, we establish further properties of
the functions $\bar g_{a,x}$ and $\bar h_{n,a,x}$.
Let $\varepsilon>0$ and set $A_2=[1-\varepsilon,1+\varepsilon]$.
In the following, we are interested in the limit $\varepsilon \to 0$ and we write
$o(1)$ for error terms that go to zero in this limit.

\medskip \noindent \emph{Properties of $\bar g_{a,x} =|x|_a\bar g_{a,\hat x}$.}
By Lemma~\ref{lem:ga}, $\bar g_{a,x}$ is convex and has unique minimum
$\bar g_{a,x}(1)=m_a|x|_a$.  In particular, $\bar g_{a,\hat x}(1)=m_a$.
Taylor expansion of $\bar g_{a,\hat x}$ about $1$ gives
\begin{equation}
\label{e:gTay}
    \bar g_{a,\hat x}(y)
    =
    m_a  + \frac{1}{2!} \bar g_{a,\hat x}''(1) (y-1)^2
    + \frac{1}{3!}g_{a,\hat x}'''(y^*)(y-1)^3
\end{equation}
for some $y^*$ between $1$ and $y$.
We see from \eqref{e:asymptotics_for_u} and Lemma~\ref{lem:ga} that as
$\varepsilon\to 0$
we have
\begin{align}
\label{e:gA2_bis}
	\bar g_{a,\hat x}''(y) &=  \bar g_{a,\hat x}''(1) (1+O(\varepsilon))
    ,
\\
	\bar g_{a,\hat x}'''(y) &=  O(a)
    ,
\end{align}
uniformly in $y \in A_2$,
in $a\le a_{\max}$, and in $x \neq 0$.
By \eqref{e:gbarprimeprime}, $\bar g_{a,\hat x}''(1) \asymp a$
uniformly in $a\le a_{\max}$ and in $x$.
For the endpoints of $A_2$, the above implies that there exists a constant $\gamma>0$
such that, for $\varepsilon$ sufficiently small
\begin{align}
\label{e:gA3}
	\bar g_{a,\hat x}(1 \pm \varepsilon)
	&= \bar g_{a,\hat x}(1) + \frac{1}{2} \bar g''_{a,\hat x}(1) \varepsilon^2
        + O(a)\varepsilon^3
        \ge m_a  + \gamma a  \varepsilon^2
,
    \\
\label{e:gpA3}
	\bar g'_{a,\hat x}(1 \pm \varepsilon)
	&= \pm \bar g''_{a,\hat x}(1)\varepsilon_n +O(a \varepsilon^2)
	= O(a  \varepsilon)
    ,
\end{align}
uniformly in $a\le a_{\max}$ and in $x \neq 0$.

\medskip \noindent \emph{Properties of $\bar h_{n,a,x}$.}
We first prove that, as $\varepsilon\to 0$,
\begin{align}
\label{e:hA2a}
	\bar h_{n,a,x}(y)
	&=  (1+o(1))\prod_{j=1}^d(1+\hat u_a^2 \hat x_j^2)^{-1/4}
    \quad
    \text{uniformly in $y \in A_2$, in $a\le a_{\max}$, and in $x\neq 0$}
    .
\end{align}
When $y \in A_2$, the ratio $y/u_a = y|x|_a/\hat u_a$ is bounded away
from zero uniformly in $a\le a_{\max}$ and $x \neq 0$.
The estimate \eqref{e:hA2a} then follows from
\eqref{e:bar_h_new} and the fact (see \cite[p.~83]{Olve97}) that
\begin{equation}
	\sqrt{2\pi n t} \bar I_0(nt) = 1+o(1)
    \text{ uniformly in $t$ bounded away from zero.}
\end{equation}
Next, we claim that there is a $C >0$ such that
\begin{align}
\label{e:hub}
	\bar h_{n,a,x}(y) &\le  Cy^{\q-1-d/2}
    \quad
    \text{ uniformly in $y> 0$, in $a\le a_{\max}$, and in $x\neq 0$}.
\end{align}
To obtain \eqref{e:hub}, we use \eqref{e:bar_h_new} and
the fact that
$\bar I_0(t) \leq O(t^{-1/2})$ which can also be seen from \cite[p.~83]{Olve97}.
Finally,
we use \eqref{e:gbarprimeprime} and \eqref{e:hA2a} to see, after some algebra, that
as   $\varepsilon\to 0$ we have
\begin{align}
\label{e:hgratio}
    \frac{\bar h_{n,a,x}(1)}{\sqrt{\bar g_{a,x}''(1)}}
    &
    =
    \frac{1+o(1)}
    {(|x|_a \bar g_{a,\hat x}''(1)
    \prod_{j=1}^d (1+ \hat u_a^2 \hat x_j^2)^{1/2})^{1/2}}
    = \frac{\kappa_a(\hat x)}{ \sqrt{|x|_a \hat u_a}}(1+o(1))
    ,
\end{align}
with
\begin{equation}
\label{e:kappadef}
    \kappa_a(\hat x)
    =
    \Big( \sum_{j=1}^d \hat x_j^2 \prod_{i\neq j} (1+\hat u_a^2 \hat x_i^2)^{1/2} \Big)^{-1/2}
    ,
\end{equation}
and where the $o(1)$ term goes to zero as
$\varepsilon\to 0$
uniformly in $a\le a_{\max}$ and $x \neq 0$.

\medskip
\begin{proof}[Proof of Theorem~\ref{thm:green-asy}(i)]
Recall from \eqref{e:CqOZII} that
\begin{align}
\label{e:Jsufficient}
     C^{(\q)}_{a}(nx)
     &= (1+\delta_n)
     \Big(\frac{n}{u_a}\Big)^{\q-d/2}
     \alpha_\q
     \int_0^\infty \bar h_{n,a,x}(y)
     e^{-n\bar g_{a,x}(y)}\D y
     .
\end{align}
With $c_{0,\q} = \sqrt{2\pi}\,\alpha_\q$
as in \eqref{e:calim} (recall \eqref{e:alphaqdef}), we define
\begin{equation}
\label{e:caqx}
    c_{a,\q,\hat x}
    =
    c_{0,\q}
    \kappa_a(\hat x)
    \Big(\frac{\hat u_a}{m_a}\Big)^{(d-1-2q)/2}
    .
\end{equation}
By definition,
$c_{a,\q,\hat x}$ depends on $x$ only via its direction $\hat x$.  Also, as stated in \eqref{e:calim},
$\lim_{a\to 0}c_{a,\q,\hat x} = c_{0,\q}$ due to \eqref{e:asymptotics_for_u},
\eqref{e:kappadef}, the relation $m_a \sim \sqrt{2d}\, a$, and the fact
that $\|\hat x\|_2 = \|x\|_2/|x|_a \to 1$ by Proposition~\ref{prop:norm}.
Our goal is to prove that \eqref{e:mr} holds,
which by \eqref{e:Jsufficient} will follow if we prove that,
uniformly in $x \neq 0$ and in $a\le a_{\max}$, as $an \to \infty$ we have
\begin{align}
         \alpha_\q
     \int_0^\infty \bar h_{n,a,x}(y)
     e^{-n\bar g_{a,x}(y)}\D y
     & \sim
     \Big(\frac{u_a}{n} \Big)^{\q-d/2}
    c_{a,\q,\hat x}
     \frac{m_a^{(d-1-2\q)/2}}{(n|x|_{a})^{(d+1-2\q)/2} }
    \;
    e^{-m_an|x|_a}
    \nnb
    & =
    c_{0,\q}
    \frac{1}{\sqrt{n |x|_a \hat u_a }}\kappa_a(\hat x)
    e^{-m_an|x|_a}
\label{e:J2want}
\end{align}
(the equality holds by definition---recall that $u_a = \hat u_a/ |x|_a$).

We set $\varepsilon_n =(an)^{-1/4}$,
which does obey  $\varepsilon_n  \to 0$ as we imposed below
\eqref{e:asymptotics_for_u}, and we
divide the interval of integration in \eqref{e:J2want} into three subintervals:
\begin{equation}
    A_1 = [0,1-\varepsilon_n],
    \qquad
    A_2= [1-\varepsilon_n,1+\varepsilon_n]
    \qquad
    A_3 = [1-\varepsilon_n, \infty).
\end{equation}
Then we set
\begin{equation}
    J_i =
    \alpha_\q
    \int_{A_i} \bar h_{n,a,x}(y) e^{-n\bar g_{a,x}(y)} dy \qquad (i=1,2,3).
\end{equation}
We will prove that $J_2$ gives the main contribution
to \eqref{e:J2want}, with $J_1$ and $J_3$ relatively small.

\medskip\noindent
\emph{The integral $J_2$.}
By \eqref{e:gTay}, \eqref{e:gA2_bis}, and \eqref{e:hA2a},
\begin{align}
	J_2
    &= (1+o(1))
    \alpha_\q
    \bar h_{n,a,x}(1)
    e^{-n m_a|x|_a}
    \int_{-\varepsilon_n}^{\varepsilon_n}\exp\Big(-\frac{n}{2} \bar g_{a,x}''(1)(1+o(1))y^2 \Big)\D y,
\end{align}
with the $o(1)$
(as $\varepsilon_n\to 0$)
uniform in $y$, in $a\le a_{\max}$, and in $x \neq 0$.
We make the change of variables $v=y(n\bar g_{a,x}''(1))^{1/2}$ and obtain,
with $M_n=\varepsilon_n(n\bar g_{a,x}''(1))^{1/2}$,
\begin{align}
	J_2
    &= (1+o(1))
    \alpha_\q
    \frac{\bar h_{n,a,x}(1)}{\sqrt{n \bar g_{a,x}''(1)}}
    e^{-n m_a|x|_a}
    \int_{-M_n}^{M_n}
    \exp\Big(-\frac{1}{2}  (1+o(1))v^2 \Big)\D v
    .
\end{align}
By our choice of $\varepsilon_n$, and by
the fact that $\bar g_{a,x}''(1)= |x|_a \bar g_{a,\hat x}''(1) \asymp a|x|_a$ (as noted above
\eqref{e:gA3}), there exists a $c>0$ such that
\begin{equation}
	M_n \geq c\varepsilon_n (n|x|_aa)^{1/2}
    \ge c(na)^{1/4} \to \infty.
\end{equation}
Since $\alpha_\q (2\pi)^{1/2} = c_{0,\q}$, this gives
\begin{equation}
    J_2
    =
    (1+o(1))
    c_{0,\q}
    \frac{\bar h_{n,a,x}(1)}{\sqrt{n \bar g_{a,x}''(1)}}  e^{-nm_a|x|_a}.
\end{equation}
To obtain the desired right-hand side of \eqref{e:J2want}, we
replace the ratio in the above using \eqref{e:hgratio}.

It remains to show that the contributions from
the integrals $J_1$ and $J_3$ are relatively small.

\medskip \noindent \emph{The integral  $J_1$.}
To show that $J_1$ is relatively small compared to $J_2$, it suffices to prove that as  $an \to \infty$
\begin{equation}
\label{e:J1sufficient}
    \int_0^{1-\varepsilon_n} \bar h_{n,a,x}(y) e^{-n\bar g_{a,x}(y)} \D y
    \le \frac{o(1)}{ \sqrt{na|x|_a}} e^{-n\bar g_{a,x}(1)}
\end{equation}
uniformly in $x\neq 0$.
Let $\alpha = \q-1-d/2$.
By the upper bound
$\bar h_{n,a,x}(y) \leq Cy^{\alpha}$ of \eqref{e:hub}, the above integral is at most
\begin{equation}
    C   \int_0^{1-\varepsilon_n} y^{\alpha} e^{-n\bar g_{a,x}(y)} \D y.
\end{equation}
If $\alpha > -1$ then we simply bound the exponential by its maximum value
to obtain an upper bound proportional to
$\exp[-n\bar g_{a,x}(1-\varepsilon_n)]$.  By \eqref{e:gA3}, this gives an
upper bound
(with $\gamma >0$)
\begin{equation}
    e^{-nm_a|x|_a} e^{-n \gamma a\varepsilon_n^2 |x|_a}
    = e^{-nm_a|x|_a}  e^{-\gamma (na)^{1/2}|x|_a}
    = o((na|x|_a)^{-1/2}) e^{-nm_a|x|_a} ,
\end{equation}
which is sufficient.

If instead $\alpha \le -1$, then we apply Lemma~\ref{lem:ga} to bound
$y^{\alpha} e^{-n\bar g_{a,x}(y)}$ by its maximum which is attained
on $[\frac 12, 1]$ and hence is at most
$2^{|\alpha|}\exp[-n\bar g_{a,x}(1-\varepsilon_n)]$, and this is again sufficient
to obtain \eqref{e:J1sufficient}.  This proves that $J_1$ is negligible compared to $J_2$.

\medskip \noindent \emph{The integral  $J_3$.}
We bound $\bar h(y)$ by $y^{\alpha}$ with $\alpha = \q-1-d/2$.
If $\alpha <-1$, so that $y^\alpha$ is integrable,
 then we simply
 extract additional
exponential decay (compared to $J_2$) using \eqref{e:gA3} again.
Then we integrate $y^\alpha$ over $[1,\infty]$ and obtain an upper bound of the form
\begin{equation}
\label{e:J3bd}
	Ce^{-\gamma (na)^{1/2}|x|_a} = o\Big( \frac{1}{\sqrt{na|x|_a}}\Big),
\end{equation}
which is sufficient for the case $\alpha<-1$.

If instead $\alpha \ge -1$ then we use integration by parts to reduce the power.
For example, if $\alpha \in (-1,0]$ then, with $t=1+\varepsilon_n$
and $f(y)=n\bar g_{a,x}(y)$,
we use the facts that
by \eqref{e:gpA3}
$f'(t) \ge  c(na)^{3/4}|x|_a$
which eventually exceeds $1$, that
$f'$ is positive and increasing on $[t,\infty)$,
and that $\lim_{y\to\infty} y^\alpha e^{-f(y)}=0$
by \eqref{e:g_at_infty}, to obtain
\begin{align}
    \int_{t}^{\infty} y^{\alpha}e^{-f(y)}\D y
    &\leq \frac{1}{f'(t)}\int_{t}^{\infty} y^{\alpha}f'(y)e^{-f(y)}\D y
    \leq t^\alpha e^{-f(t)}+ |\alpha|  \int_t^\infty y^{\alpha-1}e^{-f(y)} \D y   .
\end{align}
The term $t^\alpha e^{-f(t)}$ is bounded above by a multiple of
$e^{-f(1+\varepsilon_n)}$, which is bounded as in \eqref{e:J3bd}.
This process can be iterated to reduce the power of $y$ to below $-1$,
which we have seen to be sufficient.

\medskip \noindent
This completes the proof.
\end{proof}

\section{Continuum regime: Proof of Theorem~\ref{thm:green-asy}(ii)}
\label{sec:contlim}

In this section, we prove Theorem~\ref{thm:green-asy}(ii). The method of proof is different from the proof of Theorem~\ref{thm:green-asy}(i) and relies
instead on a dominated convergence argument which applies simultaneously
for both $\scale >0$ and $\scale = 0$.

We define the \emph{heat kernel} (for the normalised Laplacian $\frac {1}{2d}\Delta_{\R^d}$)
\begin{equation}
\label{e:ptx}
    p_t(x) =
    \Big(\frac{d}{2\pi t}\Big)^{d/2} e^{-d\|x\|_2^2 /2t}
    \qquad
    (x \in \R^d, \; t >0).
\end{equation}
Recall the definitions of $G_s(x)$ and $G_0(x)$ in \eqref{e:Gaqx-2d} and \eqref{e:G0qx}. The following representations of the continuum Green function will be useful.
For the case $a=0$ (with $d>2\q$), we  observe that the change of variables
$s=d\|x\|_2^2/2t$ leads to
\begin{align}
\label{e:G0int}
    &
    \frac{1}{\Gamma(\q)}
    \int_0^\infty \D t \, t^{\q-1}
    p_t(x)
     =
     \frac{d^\q \Gamma(\frac{d-2\q}{2})}{2^{\q}\pi^{d/2}\Gamma(\q)}
    \frac{1}{\|x\|_2^{d-2\q}} = G_0^{(\q)}(x).
\end{align}
For $a>0$, we use
\begin{align}
\label{e:ptK}
    \frac{1}{\Gamma(\q)}
    \int_0^\infty \D t \, t^{\q-1} e^{-ta^2}
    p_t(x)
     &   =
    \Big(\frac d2 \Big)^{d/2} \frac{1}{\Gamma(\q)}
    \frac{a^{d-2\q}}{2^d \pi^{d/2}}
    \int_0^\infty \D s \frac{1}{s^{(d-2\q)/2}}
    e^{-s - d(a\|x\|_2)^2/2s}
    =
    G_a(x),
\end{align}
where we applied the formula
\begin{equation}
\label{e:Kint}
    K_\alpha(z) = \frac 12 \Big(\frac z2\Big)^\alpha \int_0^\infty \frac{1}{t^{\alpha+1}}e^{-t-z^2/4t}\D t
    \qquad
    (\alpha \in \R,\; z>0)
\end{equation}
from \cite[8.432.6]{GR07} for the last equality.

\begin{proof}[Proof of Theorem~\ref{thm:green-asy}(ii)]
Fix $x \neq 0$; without loss of generality we may assume that $x_j \ge 0$ for all $j$.
We are interested in $a = \scale/n$ with some fixed $\scale\ge 0$.
We rewrite \eqref{e:Besselcont} as
\begin{align}
    n^{d-2q}  C_{{s}/n}^{(\q)}(nx)
    &=
    \frac{1}{\Gamma(\q)} \int_0^\infty \D t\, t^{\q-1}
    e^{-\scale^2 t} f_n(t)
\label{e:cl1}
\end{align}
where we define (with $x$-dependence suppressed in the notation)
\begin{equation}
\label{e:fnu}
    f_n(t) = n^{d}
    \prod_{j=1}^{d}\bar I_{n x_j} (n^2t/d)
    .
\end{equation}
Our goal is to prove that \eqref{e:cl1} has limit $G_{s}^{(\q)}(x)$ as $n\to\infty$.

We prove below in Lemmas~\ref{lem:flim}--\ref{lem:fbd} that
\begin{equation}
\label{e:flim}
    \lim_{n\to\infty} f_n(t) = p_t(x),
\end{equation}
and that there are positive constants $C,\delta$ (depending on $x,d$) such that
\begin{equation}
\label{e:fbd0}
    f_n(t) \le
    C
    \Big(
    \1_{t \le 1}
    +
    t^{-d/2}
    \1_{t > 1}
    \Big)
    \quad \text{uniformly in $n \ge 0$ and $t \ge 0$.}
\end{equation}
Once \eqref{e:flim}--\eqref{e:fbd0} are proved, since the upper bound in \eqref{e:fbd0}
is integrable after insertion in the integral on the right-hand side of \eqref{e:cl1}
(assuming $d>2\q$ if $s=0$),
the dominated convergence theorem
can be applied.
For $s>0$, this gives
\begin{align}
    \lim_{n \to\infty}
    \frac{1}{\Gamma(\q)} \int_0^\infty \D t\, t^{\q-1}
    e^{-ts^2} f_n(t)
    &=
    \frac{1}{\Gamma(\q)} \int_0^\infty \D t\, t^{\q-1}  e^{-ts^2}
    p_t(x),
\end{align}
and we have seen in \eqref{e:ptK} that the right-hand side is equal to $G_{s}^{(\q)}(x)$.
Similarly, if $s=0$ and $d>2\q$, we instead obtain
\begin{align}
    \lim_{n\to\infty}
    \frac{1}{\Gamma(\q)} \int_0^\infty \D t\, t^{\q-1}
    f_n(t)
    &=
    \frac{1}{\Gamma(\q)} \int_0^\infty \D t\, t^{\q-1}
    p_t(x)  ,
\end{align}
which is the integral identified as $G_0(x)$ in \eqref{e:G0int}.
This completes the proof.
\end{proof}

It remains to prove \eqref{e:flim}--\eqref{e:fbd0}.  We do this in  Lemmas~\ref{lem:flim}--\ref{lem:fbd}, whose proofs rely on
the asymptotic form of the modified Bessel function from Lemma~\ref{lem:Inunu2}.

\begin{lemma}
\label{lem:flim}
For $d\ge 1$, for $x \in \Z^d$ with $x_j \ge 0$, and for $t>0$,
\begin{equation}
\label{e:flim-lem}
    \lim_{n\to\infty} f_{n}(t) = p_t(x)  .
\end{equation}
\end{lemma}

\begin{proof}
Recall the definition of $f_n(t)$ in \eqref{e:fnu}.
When $x_j >0$, it follows from \eqref{e:Inunu}
that
\begin{align}
\label{e:Lasy}
     \bar I_{n x_j} (n^2t/d)
     = \bar I_{n x_j}((n x_j)^2 t/ d x_j^2)
     &\sim
     \frac{1}{n} \Big(\frac{d}{2\pi t}\Big)^{1/2}e^{-dx^2_j/2t}
     .
\end{align}
If $x_j =0$ then, since
$\bar I_0(z)\sim (2\pi z)^{-1/2}$  as $z \to \infty$,
\begin{align}
\label{e:I0asy}
    \bar I_0 (n^2t/d)
    &\sim \frac{1}{n} \Big(\frac{d}{2\pi t}\Big)^{1/2}
    ,
\end{align}
which is the same formula as \eqref{e:Lasy} but with $x_j$ set equal to zero.
Substitution of \eqref{e:Lasy}--\eqref{e:I0asy} into \eqref{e:fnu},
together with the definition of $p_t(x)$ in \eqref{e:ptx}, then gives
\begin{equation}
    \lim_{n \to\infty} f_n(t) = p_t(x)
\end{equation}
and the proof is complete.
\end{proof}

\begin{lemma}
\label{lem:fbd}
Let $d \ge 1$ and
let $x \in \Z^d$ be nonzero with $x_1 \ge x_2 \ge \cdots \ge x_d \ge 0$.
There
are constants $C,\delta,n_0 >0$ (depending only on $d$)
such that
\begin{equation}
\label{e:fbd}
    f_n(t) \le
    C
    \Big(
    \1_{t \le 1}
    +
    t^{-d/2}
    \1_{t > 1}
    \Big)
    \quad \text{uniformly in $n \ge n_0$ and $t \ge 0$.}
\end{equation}
\end{lemma}

\begin{proof}
We use $C$ to denote a constant that may depend on $d$ and may
change value from line to line.
By hypothesis, $x_1 \ge 1$.
Since $I_{\alpha'}(z) < I_\alpha(z)$ for any $z \ge 0$ and any $\alpha' > \alpha \ge 0$
\cite{Coch67}, we can bound each factor with $j \ge 2$ in \eqref{e:fnu} above by
$\bar I_{0}(n^2t/d)$ to obtain
\begin{equation}
    f_n(t) \le n^{d-1}\Big(\bar I_{0}(n^2t/d)\Big)^{d-1} n \bar I_{n}(n^2t/d).
\end{equation}
Since $\bar I_0(z)\sim (2\pi z)^{-1/2}$ as $z \to \infty$,
and since $\bar I_0(z) \le 1$ for all $z \ge 0$, we see that
\begin{equation}
\label{e:fgbd}
    f_n(t) \le
    C \min (n^{d-1}, t^{-(d-1)/2}) \;
    n \bar I_{n} (n^2t/d  ).
\end{equation}
By \eqref{e:Lclaim}, there exist $C,\delta,n_0 >0$ such that
\begin{align}
\label{e:fgbd2}
    n\bar I_{n}(n^2 t/d)
    &\le C
    \Big(
    n e^{ - \delta n} \1_{ t < d/2n}
    +
    t^{-1/2} e^{- \delta/t}
    \1_{t \geq d/2n} \Big)
    \qquad
    (n \ge n_0,\; t>0).
\end{align}
We insert \eqref{e:fgbd2} into \eqref{e:fgbd}, using $n^{d-1}$ for the first
term and $t^{-(d-1)/2}$ for the second, and obtain
\begin{equation}
	f_n(t) \le C \Big(
    n^d e^{ - \delta n} \1_{t < d/2n}
    +
    t^{-d/2} e^{- \delta/t}
    \Big)
    \qquad
    (n \ge n_0,\; t>0).
\end{equation}
The second term on the right-hand side is bounded for $t \le 1$ and is less than
$t^{-d/2}$ for $t >1$.  Also, $n^d e^{ - \delta n}$ is bounded
as a function of $n$, and $\1_{t<d/2n} \le \1_{t \le 1}$ once
$n \ge d/2$
so
the first term is bounded by a multiple of $\1_{t \le 1}$.
This completes the proof.
\end{proof}

\section{Properties of the norm:  Proof of Proposition~\ref{prop:norm}}
\label{sec:norm}

In this section, we prove Proposition~\ref{prop:norm}.  We assume throughout that $d \ge 1$
and $a>0$.
Recall the definition
\begin{equation}
\label{e:normdef5}
    |x|_a = \frac{1}{m_a} \sum_{i=1}^d x_i \arcsinh(x_iu)
    \qquad (x \neq 0).
\end{equation}
Proposition~\ref{prop:norm} asserts that
$|\cdot|_a$ defines a norm on $\R^d$ which is monotone increasing
in $a$ and for all $x\in\R^d$ obeys
\begin{equation}
|x|_a=\|x\|_2(1+O(a^2)),
\qquad \lim_{a\to \infty}|x|_a  = \|x\|_1,
\end{equation}
with the error term in the first equality uniform in nonzero $x$ as $a \to 0$.
From this, we conclude immediately that $\|x\|_2 \le |x|_a \le \|x\|_1$.

For the limit $a \to 0$,
it follows from the relation $u_a(x)  = \frac{\sqrt{2d}a}{\|x\|_2}(1+O(a^2))$
from \eqref{e:asymptotics_for_u}, together
with $m_a=\sqrt{2d}a(1+O(a^2))$ from \eqref{e:maasy}
and the definition \eqref{e:normdef5},
that
\begin{align}
	|x|_a
	&= \frac{1}{m_a}\sum_{i=1}^d x_i^2\frac{\sqrt{2d}a}{\|x\|_2}
    \big(1+O(a^2)+O(a^2 x_i^2\|x\|_2^{-2})\big)
    = \|x\|_2(1+O(a^2)).
\label{e:a2}
\end{align}
To see that $\lim_{a\to\infty}|x|_a = \|x\|_1$, we first observe that $m_a = \arccosh(1+da^2)\sim \log a^2$ as $a\to\infty$.  Also, it follows from \eqref{e:udef1} that
$u=u_a(x) \sim \|x\|_1^{-1}da^2$ as $a \to \infty$, and therefore
\begin{equation}
    |x|_a
    \sim   \frac{1}{\log a^2}  \sum_{i=1}^d |x_i| \log a^2 = \|x\|_1.
\end{equation}
Thus, to complete the proof of Proposition~\ref{prop:norm}, it suffices to prove
that $|\cdot|_a$ defines a norm on $\R^d$, and that $|x|_a$ is monotone increasing in $a$
for each fixed $x$.  We prove these two items in Lemmas~\ref{lem:normpf}--\ref{lem:normincreasing}.
To lighten the notation, we will write $C_a(x)$ instead of $C_a^{(1)}(x)$.

The following elementary lemma is the basis for our proof of the triangle inequality for $|\cdot|_a$.

\begin{lemma}
\label{lem:Cxy}
For $d \ge 1$, for $x,y\in \Z^d$ and for $a > 0$ (also for $a=0$ if $d >2$),
\begin{equation}
\label{e:CCCC}
    C_a(0)C_a(x) \ge C_a(y) C_a(x-y).
\end{equation}
\end{lemma}

\begin{proof}
Let $p_n(x)$ be the $n$-step transition probability for simple random walk (without killing)
to travel from
$0$ to $x$ in $n$ steps, and
let $P_\kappa(x) = \sum_{n=0}^\infty (1-\kappa)^{n} p_n(x)$.
We have seen below \eqref{e:Gkappadef} that $(1+a^2)C_a(x)=P_\kappa(x)$ with
$\kappa = \frac{a^2}{1+a^2}$,
so it suffices to prove \eqref{e:CCCC} instead for $P_\kappa$.

Let $q_n(x)$ be the $n$-step transition probability for simple random walk to travel from
$0$ to $x$ in $n$ steps without revisiting $0$, and let
$Q_\kappa(x) = \sum_{n=0}^\infty (1-\kappa)^{n} q_n(x)$.
By considering only walks from $0$
to $x$ which pass through a fixed $y \in \Z^d$ and visit $y$ for the last time at
the $m^{\rm th}$ step, we obtain
\begin{equation}
    p_n(x)      \ge \sum_{m=0}^n p_m(y)q_{n-m}(x-y).
\end{equation}
This inequality gives
\begin{equation}
    P_\kappa(x) \ge P_\kappa(y)Q_\kappa(x-y).
\end{equation}
Also, with $m$ the time of the last return to $0$,
\begin{equation}
    p_n(x) = \sum_{m=0}^n p_m(0)q_{n-m}(x),
\end{equation}
and by replacing $x$ with $x-y$ we similarly obtain
\begin{equation}
    P_\kappa(x-y) = P_\kappa(0)Q_\kappa(x-y).
\end{equation}
Therefore,
\begin{equation}
    P_\kappa(0)P_\kappa(x) \ge P_\kappa(y)P_\kappa(x-y),
\end{equation}
and the proof is complete.
\end{proof}

\begin{lemma}
\label{lem:normpf}
For $d \ge 1$ and $a>0$, $|\cdot|_a$ is a norm on $\R^d$.
\end{lemma}

\begin{proof}
By its definition in \eqref{e:normdef5}, $|\cdot|_a$ is non-negative and homogeneous
(recall that $u_a(\lambda x) = |\lambda|^{-1}u_a(x)$),
with $|x|_a=0$ if and only if $x=0$.  It remains only to prove the triangle inequality.

To prove the triangle inequality first for points in $\Z^d$, we conclude
from Lemma~\ref{lem:Cxy} that
\begin{equation}
    C_a(0)C_a(nx) \ge C_a(ny)C_a(nx-ny)
    \qquad
    (x,y\in\Z^d).
\end{equation}
The asymptotic formula \eqref{e:mr}
(whose proof did not use the triangle inequality we are now proving)
implies that
\begin{equation}
    - \lim_{n \to \infty} \frac 1n \log C_a(nx) = m_a |x|_a.
\end{equation}
From this, we obtain
\begin{equation}
    m_a|x|_a \le m_a|y|_a + m_a |x-y|_a,
\end{equation}
and hence the triangle inequality does hold when the norm is evaluated at points in $\Z^d$.

For $x,y\in \R^d$, we write $\floor{x} = (\floor{x_1},\cdots, \floor{x_d})$. The triangle inequality holds for $\floor{2^nx},\floor{2^ny}$, for all $n\in \N$. By homogeneity, it also holds for $\frac{\floor{2^nx}}{2^n},\frac{\floor{2^ny}}{2^n}$. Since $x \mapsto u(x)$ is a continuous function on $\R^d\setminus\{0\}$, $x\mapsto|x|_a$ is a continuous function on $\R^d$ which is extended continuously at $0$ by $|0|_a = 0$. Thus by letting $n\to\infty$ we obtain the triangle inequality for all $x,y\in\R^d$.
\end{proof}

\begin{lemma}
\label{lem:normincreasing}
For $d \ge 1$, $a>0$, and $x\in \R^d$,
the norm $|x|_a$ is a monotone increasing function of $a$.
\end{lemma}

\begin{proof}
We fix $x\in \R^d$ and prove that the function $f(a)=|x|_a$
is increasing in $a$.  It is convenient to introduce the notation
\begin{equation}
    \sigma_i = \sigma_i(x) = x_i^2 (1+x_i^2u^2)^{-1/2},
    \qquad
    \|\sigma\|_1 = \sum_{i=1}^d \sigma_i.
\end{equation}
Implicit differentiation of $\cosh m_a=1+da^2$ with respect to $a$ gives
\begin{equation}
    m_a'(a) = \frac{2da}{\sinh m_a},
\end{equation}
and differentiation of  \eqref{e:udef1}   leads to
\begin{equation}
    u_a'  = \frac{2da}{u  \|\sigma\|_1}.
\end{equation}
Therefore, by \eqref{e:normdef5},
\begin{align}
    f'(a) &=
    - \frac{m_a'}{m_a^2}\sum_{i=1}^d x_i \arcsinh(x_iu)
    +
    \frac{u'}{m_a} \|\sigma\|_1
\label{e:derivative_a_norm}
    =
    \frac{2da}{um_a }\Big(1-\frac{|x|_a u }{\sinh m_a }\Big).
\end{align}

Let
\begin{equation}
    U(x) = |x|_a u_a(x) \quad (x\in \R^d \setminus \{0\}),
\end{equation}
and note that when $x=e_i$
is a unit vector, it follows from \eqref{e:ue}--\eqref{e:norm1} that
\begin{equation}
    U(e_i) = u_a(e_i) = \sinh m_a.
\end{equation}
Thus it suffices to show that $U$ is maximal at $e_i$, as this implies $f'(a) \ge 0$.
Since $U(\lambda x)=U(x)$ for all $\lambda >0$, and since
$U$ is continuously differentiable and bounded  on $\R^d\setminus \{0\}$,
the maximum exists and
will be attained along lines through the origin.  There will be several lines
since $U(x)$ is invariant under permutation or sign changes of the coordinates of $x$, so we may restrict attention to nonzero $x$ with $x_1 \geq \cdots \geq x_d$.

We first argue that any critical point $x^*$ of $U$ must have all its nonzero coordinates
equal.  A critical point of $U$ obeys
\begin{align}
\label{e:xcrit}
\frac{\partial U}{\partial x_i}(x^*) &= \frac{u\arcsinh(x^*_iu)}{m_a}+|x^*|_a\frac{\partial u}{\partial x_i}(x^*)  = 0 \qquad (i=1,\ldots,d).
\end{align}
Differentiation of \eqref{e:udef1} with respect to $x_i$ gives
\begin{align}
    x_i \frac{\partial u}{\partial x_i} &=
    -u\frac{\sigma_i}{\|\sigma\|_1}
.
\end{align}
Thus, with $\sigma_i^*=\sigma_i(x^*)$ and $u^*=u_a(x^*)$, \eqref{e:xcrit} can be rephrased as
\begin{align}
\label{e:crit_point_U_norm}
\frac{\sigma^*_i}{\|\sigma^*\|_1}
&= \frac{x^*_i \arcsinh u^* x^*_i}{m_a |x^*|_a}
\qquad (i=1,\ldots,d).
\end{align}
Let $k\ge 1$ denote the largest subscript $i$ such that $x^*_i>0$.
From \eqref{e:crit_point_U_norm}, we see that
\begin{align}
\label{e:crit_point_U_norm_bis}
	\frac{\sigma_i^*u^{*2}}{u^*x^*_i\arcsinh u^* x^*_i}
    &= \frac{\sigma_j^*u^{*2}}{u^*x^*_j\arcsinh u^* x^*_j}
    \qquad ( i,j \le k ).
\end{align}
An elementary calculation shows that the
function
\begin{equation}
    t \mapsto \frac{t}{\sqrt{1+t^2} \arcsinh t}
\end{equation}
is a bijection from $\R^+$ onto $[0,1]$.  Thus \eqref{e:crit_point_U_norm_bis} implies the equality $u^*x^*_i = u^*x^*_j$, so indeed all nonzero coordinates of any critical vector $x^*$ must be equal.

Let $v_k = \sum_{i=1}^k e_i$, for $k=1,\ldots,d$.  It remains only to determine
which value of $k$ maximises $U(v_k)$.
The explicit values of $u_a(v_k)$ and $|v_k|_a$ can be computed from \eqref{e:udef1}--\eqref{e:normdef5}, with the result that
\begin{align}
    u_a(v_k) & = \sqrt{(1+da^2/k)^2-1} = \sinh(\arccosh(1+da^2/k)),
    \\
    |v_k| &= \frac{1}{m_a}k\arcsinh u_a(v_k) = \frac{1}{m_a}k \arccosh(1+da^2/k).
\end{align}
From this, we find that
\begin{align}
    U(v_k) & = \frac{da^2}{m_a} \psi(1+da^2/k),
    \qquad
    \psi(x) = \Big(\frac{x+1}{x-1} \Big)^{1/2} \arccosh x.
\end{align}
A computation gives
\begin{align}
	\psi'(x) &= \frac{1}{x-1}\Big(1-\frac{\arccosh x}{\sqrt{x^2-1}}\Big)
    \ge 0,
\end{align}
with the inequality due to the fact that
 $\arccosh x \leq \sqrt{x^2-1}$ for all $x\ge 1$.
Therefore $U(v_k)$ is decreasing in $k$ and the maximum of $U(v_k)$ is attained at $v_{1}=e_1$.  We have noted previously that this suffices, so the proof is
complete.
\end{proof}

\appendix
\section{Continuum Green function}
\label{sec:continuumGreen}

Let $\Delta_\R$ denote the Laplace operator for functions on $\R^d$,
normalised by $\frac{1}{2d}$, for dimensions $d \ge 1$.
Let $\scale\ge 0$ and $\q>0$.
The Green function
(or \emph{fundamental solution})
of the operator $-\Delta_\R +\scale^2$ corresponds to the $\q=1$ case of
the integral
\begin{equation}
\label{e:Gxq}
    G_\scale^{(\q)}(x)
    = \int_{\R^d} \frac{e^{ik\cdot x}}{(\frac{1}{2d}\|k\|_2^2+\scale^2)^{\q}} \frac{\D k}{(2\pi)^d}
    \qquad
    (x \in \R^d \setminus \{0\}).
\end{equation}
However this integral requires some interpretation, as it is not absolutely convergent
for large $k$ unless $d<2\q$, and for $\scale=0$ it is not convergent at $k=0$ unless
$d>2\q$.  The interpretation is in terms of tempered distributions in the
next proposition.
For the case $\scale=0$, see \cite[Theorem~2.4.6]{Graf14} for an extension
without the restriction that $d>2\q$.

\begin{prop}
\label{prop:G}
Let $d \ge 1$, $\q >0$, and $\scale \ge 0$ (with the restriction $d>2\q$ if $\scale =0$).
In the sense of tempered distributions,
the Fourier transform of $G_\scale^{(\q)}(x)$ defined by \eqref{e:Gscaling}--\eqref{e:G0qx}
is
\begin{equation}
\label{e:Ghatq}
    \hat G_\scale^{(\q)}(k) = \frac{1}{(\frac{1}{2d}\|k\|_2^2+\scale^2)^\q}.
\end{equation}
\end{prop}

For $f \in L^1(\R^d)$, we define
\begin{equation}
    \hat f(k) = \int_{\R^d} f(x) e^{ik\cdot x} \D x
    \qquad
    (k \in \R^d).
\end{equation}
In particular, $\hat p_t(k) = e^{-t\|k\|_2^2/2d}$, where $p_t$ is the heat kernel
defined in \eqref{e:ptx}.

\begin{proof}[Proof of Proposition~\ref{prop:G}]
Let $d \ge 1$, $\q>0$ and $\scale \ge 0$, with the additional assumption that $d>2\q$ if
$\scale=0$.
Let $\varphi$ be a Schwartz-class test function on $\R^d$.
In the sense of tempered distributions, the statement that
the Fourier transform is given by \eqref{e:Ghatq}
is the
statement that
\begin{equation}
\label{e:FTdistribution}
    \int_{\R^d}
    \frac{1}{(\frac{1}{2d}\|k\|_2^2+\scale^2)^\q}
    \hat\varphi(k)   \frac{\D k}{(2\pi)^d}
    =
    \int_{\R^d} G_\scale^{(\q)}(x) \varphi(x) \D x.
\end{equation}

To prove \eqref{e:FTdistribution}, we use \eqref{e:reciprocal} and Fubini's Theorem to obtain
\begin{align}
    \int_{\R^d}
    \frac{1}{(\frac{1}{2d}\|k\|_2^2+\scale^2)^\q}
    \hat\varphi(k)  \frac{\D k}{(2\pi)^d}
    & =
    \frac{1}{\Gamma(\q)}
    \int_0^\infty \D t \, t^{\q-1} e^{-t \scale^2 }
    \int_{\R^d}  \frac{\D k}{(2\pi)^d}
    e^{-t\|k\|_2^2/2d}  \hat\varphi(k) .
\label{e:Gpf1}
\end{align}
Fubini's Theorem indeed applies since the integral on the right-hand side is absolutely
convergent, because
the $t$-integral is bounded uniformly in $k$ when $\scale>0$,
and is $O(\|k\|_2^{-2\q})$ when $\scale=0$
so there is no divergence at $k=0$ when $d>2\q$ (of course there is no divergence as
$k\to \infty$ because $\hat\varphi$ is a Schwartz function).
By Parseval's relation, the last integral in \eqref{e:Gpf1} is equal to
\begin{equation}
    \int_{\R^d}  \frac{\D k}{(2\pi)^d}
    \hat{p}_t(k) \hat\varphi(k)
    =
    \int_{\R^d}  p_t(x)\varphi(x) \D x.
\end{equation}
A second application of Fubini's Theorem (justified below) then gives
\begin{align}
\label{e:Fubini}
    \int_{\R^d}
    \frac{1}{(\frac{1}{2d}\|k\|_2^2+\scale^2)^\q}
    \hat\varphi(k)  \frac{\D k}{(2\pi)^d}
    & =
    \int_{\R^d}
    \D x \, \varphi(x )
    \frac{1}{\Gamma(\q)}
    \int_0^\infty \D t \, t^{\q-1} e^{-t \scale^2 } p_t(x)
    \nnb & =
     \int_{\R^d}
    \D x \, \varphi(x )
    G_\scale(x),
\end{align}
where we used \eqref{e:G0int}--\eqref{e:ptK} for the last equality.
To justify the application of Fubini's Theorem in \eqref{e:Fubini},
it suffices to prove that the integral on its first right-hand side  is
absolutely convergent.  We have just shown that this integral has
integrand $\varphi(x) G_\scale^{(q)}(x)$.  There is no issue for large $x$, since $G_\scale(x)$ decays
as $x \to \infty$ and $\varphi$ is a Schwartz function.  As $x \to 0$, $G_\scale^{(q)}(x)$ is
asymptotically a multiple of $\|x\|_2^{d-2\q}$ for all $\scale \ge 0$
(recall \eqref{e:Kasy} for asymptotics of $K_{(d-2\q)/2}$ when $\scale>0$), and this is integrable.
This completes the proof.
\end{proof}

\section{Bessel function with large order and large argument}
\label{sec:Besselpf}

We now prove Lemma~\ref{lem:Inunu2}, which we restate here for convenience
as Lemma~\ref{lem:Inunu2-bis}.

\begin{lemma}
\label{lem:Inunu2-bis}
As $\nu \to \infty$,
\begin{equation}
\label{e:ILasy-bis}
	\bar I_\nu(\nu t) = L_\nu (t)(1+o(1))
\end{equation}
where the $o(1)$ is uniform in $t>0$.
Also, as $\nu \to \infty$, for any $s >0$,
\begin{align}
\label{e:Inunu-bis}
    \bar I_\nu(\nu^2 s) &\sim \frac{e^{-1/2s}}{\nu (2\pi s)^{1/2}},
\end{align}
with an error that is not uniform in $s$.
Finally, there exist $C,\delta, \nu_0 >0$ such that
\begin{align}
\label{e:Lclaim-bis}
    \bar I_{\nu}(\nu^2 s)
    &\le C
    \Big(
    e^{ - \delta\nu} \1_{2\nu s < 1}
    +
    \nu^{-1} s^{-1/2} e^{- \delta/s} \1_{2\nu s \ge 1}
    \Big)
    \qquad
    (\nu \ge \nu_0,\; s>0).
\end{align}
\end{lemma}

\begin{proof}
The uniform asymptotic formula \eqref{e:ILasy-bis} is given in
\cite[(7.18), p.~378]{Olve97}.

To prove \eqref{e:Inunu-bis}, we fix $s>0$ and set $t=\nu s$ in \eqref{e:ILasy-bis}.
By the definitions of $L,\psi$ in \eqref{e:L_psi_def} we have
\begin{align}
    L_\nu(\nu s) \sim \frac{e^{\nu\psi(\nu s)}}{\nu (2\pi s)^{1/2}}
\end{align}
and
\begin{align}
    e^{\nu \psi(\nu s)} & =
    e^{-\nu^2s}
    e^{\nu\sqrt{1+\nu^2 s^2}} \Big( \frac{\nu s}{1+\sqrt{1+\nu^2 s^2}}\Big)^\nu
    \sim
    e^{-1/2 s},
\end{align}
so \eqref{e:Inunu-bis} then follows from the uniformity in \eqref{e:ILasy-bis} together with
\begin{align}
    L_\nu(\nu s) \sim \frac{e^{-1/2 s}}{\nu (2\pi s)^{1/2}}.
\end{align}

To prove \eqref{e:Lclaim-bis}, we use \eqref{e:ILasy-bis} to see that there is a $\nu_0>0$ such
that
\begin{equation}
    \bar I_\nu(\nu s) \le 2 L_\nu(\nu s) \qquad (\nu \ge \nu_0,\; s>0).
\end{equation}
By the definition of $L_\nu$
in \eqref{e:L_psi_def},
\begin{align}
\label{e:Lxi-bis}
    L_\nu(t) & = \frac{e^{-\nu t +\nu \sqrt{1+t^2}}}{(2 \pi \nu)^{1/2}(1+t^2)^{1/4} }
    h_\nu(t),
    \qquad
    h_\nu(t) =
    \Big(\frac{t}{1+\sqrt{1+t^2}} \Big)^\nu .
\end{align}
We need an estimate for \eqref{e:Lxi-bis} when $t=\nu s$.
We use
\begin{align}
    \sqrt{1+t^2} & \le t +  \min (1,(2t)^{-1}),
    \\
    (1+t^2)^{-1/4} & \le \min(1, t^{-1/2}),
\end{align}
to obtain
\begin{equation}
    L_\nu(\nu s)
    \le
    C   \nu^{-1/2} \min(1,(\nu s)^{-1/2})
    e^{\min(\nu,(2s)^{-1})}
    h_\nu(\nu s).
\end{equation}
If $2\nu s \ge 1$ then  we use (with $C$ possibly changing from line to line)
\begin{equation}
    L_\nu(\nu s)
    \le
    C   \nu^{-1} s^{-1/2}
    e^{1/(2s)} h_\nu(\nu s)
    \qquad
    (2\nu s \ge 1),
\end{equation}
while if $2\nu s <1$ we use simply
\begin{equation}
    L_\nu(\nu s)
    \le
    C
    e^{\nu} h_\nu(\nu s)
    \qquad
    (2\nu s < 1).
\end{equation}

Thus it remains to prove that there exists a $\delta>0$ such that
\begin{equation}
\label{e:hnubd}
    h_\nu(\nu s) \le
    \begin{cases}
        e^{-(1+\delta)\nu} & (2\nu s < 1)
        \\
        e^{-(\delta+ 1/2)/s}& (2\nu s \ge 1).
    \end{cases}
\end{equation}
Suppose first that $2\nu s < 1$.  The inequality  \eqref{e:hnubd} holds in this case because
\begin{equation}
    h_\nu(\nu s) \le 4^{-\nu} \qquad
    (2\nu s < 1).
\end{equation}
Finally, if $2\nu s \ge 1$ then, since the function
$\tau(y )= (\frac{y}{1+\sqrt{1+y^{2}}})^{y}$ is decreasing in $y$,
\begin{align}
     h_\nu(\nu s) &\leq \tau(1/2)^{1/s}
     =
     e^{-(\log (2+\sqrt{5}))/2s}\qquad
    (2\nu s \ge 1)
     ,
\end{align}
and this suffices for \eqref{e:hnubd} because $\log(2+\sqrt{5}) >1$.
\end{proof}

\section*{Acknowledgement}
This work was supported in part by NSERC of Canada.
We thank Yacine Aoun
for helpful comments on an earlier version of this paper,
and an anonymous referee for valuable suggestions.

%\bibliography{../../../bibdef/bib}
%\bibliographystyle{plain}

\end{document}